\documentclass{amsart}
\calclayout
 \usepackage{amssymb,amsmath,amsfonts,epsfig,latexsym,tikz}
   \usepackage[alphabetic]{amsrefs}
 \usepackage{tikz-cd}
\usepackage{hyperref}
\usepackage{enumerate}
\usepackage{mathtools}
\usepackage{verbatim}
\usepackage{cleveref}
\usepackage[shortlabels]{enumitem}
\usepackage{subcaption}

\usetikzlibrary{positioning}
\usetikzlibrary{matrix}
\usetikzlibrary{decorations}
\usetikzlibrary{decorations.pathreplacing, decorations.pathmorphing, angles,quotes}
 
 \newtheorem{theorem}{Theorem}[section]
\newtheorem{definition}[theorem]{Definition}
\newtheorem{proposition}[theorem]{Proposition}
\newtheorem{lemma}[theorem]{Lemma}
\newtheorem{corollary}[theorem]{Corollary}
\newtheorem{conjecture}[theorem]{Conjecture}

\theoremstyle{definition}
\newtheorem{remark}[theorem]{Remark}
\newtheorem{example}[theorem]{Example}

\begin{document}

\title{Hodge theory for combinatorial projective bundles}          
    
\author{Matt Larson and Ethan Partida}
\date{\today}
\begin{abstract}
We prove the Hard Lefschetz theorem and Hodge--Riemann relations for certain rings which resemble the cohomology rings of projectivizations of globally generated vector bundles over toric varieties. This proves new cases of the standard conjectures of Hodge type and gives Bloch--Gieseker-type results for tautological classes of matroids. 
\end{abstract}
 
\maketitle

\vspace{-20 pt}

\section{Introduction}

Let $X$ be a proper toric variety over a field $k$ associated to a fan $\Sigma$, and let $\mathcal{E}$ be a globally generated toric vector bundle of rank $r$ on $X$. Then, as we will explain, the Chern classes of $\mathcal{E}$ are controlled by some combinatorial data associated to $\mathcal{E}$. Let $T$ be the torus embedded in $X$, and choose a decomposition of $H^0(X, \mathcal{E})$ into a direct sum of $N$ $1$-dimensional representations. This induces a surjection $\mathcal{O}^{\oplus N} \to \mathcal{E}$, where $\mathcal{O}^{\oplus N}$ is given the linearization so that this is a map of toric vector bundles. The dual of this map realizes $\mathcal{E}^{\vee}$, the dual of $\mathcal{E}$, as a subbundle of a trivial bundle. We have an induced map from $X$ to $\operatorname{Gr}(r, N)$, the Grassmannian of $r$-dimensional subspaces of $k^N$. The action of $T$ on $k^N$ induces an action of $T$ on $\operatorname{Gr}(r,N)$ so that this map is $T$-equivariant. The vector bundle $\mathcal{E}$ is the pullback of the dual of the universal subbundle $\mathcal{S}$ from $\operatorname{Gr}(r,N)$. 

To give a combinatorial description of the Chern classes of $\mathcal{E}$, we first consider the case when $T = \mathbb{G}_m^N$ and the action on $k^N$ is given by scaling the coordinates. It is convenient to generalize to the case when $T$ is a torus acting on $X$ with a dense orbit; this amounts to allowing $\Sigma$ to have a nontrivial lineality space. 
Let $L$ be the fiber over the identity of $\mathcal{E}^{\vee}$, viewed as a subspace of $k^N$. Then $L$ is a realization of a \emph{matroid} $\mathrm{M}$, which encodes the vanishing and non-vanishing of the Pl\"{u}cker coordinates of $L$. The matroid $\mathrm{M}$ of $L$ is a collection of subsets of $[N] = \{1, \dotsc, N\}$ of size $r$ called bases: a set $B$ is a basis if the composition $L \hookrightarrow k^N \twoheadrightarrow k^B$ is an isomorphism. The matroid polytope $P(\mathrm{M})$ is the polytope in $\mathbb{R}^N$ given by 
$$\text{Convex Hull}\left(\sum_{i \in B} \mathbf{e}_i : B \text{ basis of }\mathrm{M}\right).$$
The key property that $P(\mathrm{M})$ satisfies is that each edge is parallel to a vector of the form $\mathbf{e}_i - \mathbf{e}_j$. In general, a matroid on $[N]$ is a nonempty collection of subsets of $[N]$ such that the matroid polytope, as defined above, has all edges parallel to a vector of the form $\mathbf{e}_i - \mathbf{e}_j$. We say that a matroid $\mathrm{M}$ is \emph{realizable} over a field $k$ if it is given by the set of non-vanishing Pl\"{u}cker coordinates of some subspace of $k^N$. 

The image of $X$ in $\operatorname{Gr}(r,N)$ is the $T$-orbit closure of the point $[L]$ of $\operatorname{Gr}(r,N)$. This $T$-orbit closure is the toric variety of the normal fan of $P(\mathrm{M})$ \cite{Gelfand1987a}, see also \cite[Appendix A]{Speyerg}. The vector bundle $\mathcal{E}$ is pulled back from this $T$-orbit closure, so it suffices to compute its Chern classes there. 

Let $\Sigma_{P(\mathrm{M})}$ be the normal fan of $P(\mathrm{M})$, which we view as a subfan of $\mathbb{R}^N$ with a nontrivial lineality space. Each piecewise polynomial function on $\mathbb{R}^N$, i.e., each continuous function whose restriction to each cone of $\Sigma_{P(\mathrm{M})}$ is equal to a polynomial with integer coefficients, defines an element of the Chow cohomology ring of the $T$-orbit closure of $[L]$, see \cite{PayneCohomology}. On $\Sigma_{P(\mathrm{M})}$, there are some distinguished piecewise polynomial functions $\mathfrak{c}_0 = 1$, $\mathfrak{c}_1, \dotsc, \mathfrak{c}_r$. On the cone of $\Sigma_{P(\mathrm{M})}$ corresponding to a basis $B = \{i_1, \dotsc, i_r\}$ of $\mathrm{M}$, $\mathfrak{c}_i$ is the $i$th elementary symmetric function of $t_{i_1}, \dotsc, t_{i_r}$; that this is a continuous function follows from the fact that all edges of $P(\mathrm{M})$ are parallel to vectors of the form $\mathbf{e}_i - \mathbf{e}_j$. These piecewise polynomial functions correspond to the Chern classes of the restriction of $\mathcal{S}$ to the $T$-orbit closure of $[L]$, see \cite[Section 3]{BEST}. 
In particular, the Chern classes of the restriction of $\mathcal{S}$ only depend on the matroid of $L$. 

In the general case, the action of $T \cong \mathbb{G}_m^m$ on $k^N$ induces a map $\varphi \colon T \to \mathbb{G}_m^N$. Let $\varphi^* \colon \mathbb{Z}^N \to \mathbb{Z}^m$ be the induced map on character lattices. That there is a $T$-equivariant map $X \to \operatorname{Gr}(r, N)$ which factors through the toric variety of $P(\mathrm{M})$ means that the normal fan of $\varphi^*(P(\mathrm{M}))$ is a coarsening of the fan of $X$. Let $N$ denote the cocharacter lattice of $T$, and set $N_{\mathbb{R}} \coloneqq N \otimes_{\mathbb{Z}} \mathbb{R}$. The induced map $\varphi_* \colon N_{\mathbb{R}} \to \mathbb{R}^N$ gives a map of fans from the fan of $X$ to $\Sigma_{P(\mathrm{M})}$, and so the pullbacks of $\mathfrak{c}_1, \dotsc, \mathfrak{c}_r$ are piecewise polynomials on the fan of $X$. The Chow cohomology classes associated to these piecewise polynomials are the Chern classes of $\mathcal{E}^{\vee}$. 

In summary, any globally generated $T$-equivariant vector bundle on $X$ is pulled back from the toric variety associated to the matroid polytope of a matroid which is realizable over $k$. 
In particular, classifying the possible Chern classes of globally generated $T$-equivariant vector bundles on a proper toric variety over $k$ reduces to understanding the combinatorics of matroids which are realizable over $k$. 

\medskip

In this paper, we will prove some positivity properties for the classes discussed above. We will show that certain rings constructed using them have a collection of properties known as the \emph{K\"{a}hler package}. 
Let $A = A^0 \oplus \dotsb \oplus A^n$ be a finite-dimensional commutative graded $\mathbb{R}$-algebra which is equipped with a degree map $\deg_A \colon A^n \to \mathbb{R}$ and a nonempty open convex cone $\mathcal{K} \subset A^1$. We say that $A$ has the K\"{a}hler package if the following properties hold. 
\begin{enumerate}
   \item (Poincar\'{e} duality) For all $i$, the pairing 
   $$A^i \times A^{n - i} \to \mathbb{R}, \quad (x, y) \mapsto \deg_A(xy)$$
   is nondegenerate. 
   \item\label{item:HL} (Hard Lefschetz) For every $0 \le i \le n/2$ and $\ell \in \mathcal{K}$, the map
   $$A^i \to A^{n - i}, \quad x \mapsto \ell^{n - 2i} x$$
   is an isomorphism.
   \item\label{item:HR} (Hodge--Riemann) For every $0 \le i \le n/2$ and $\ell \in \mathcal{K}$, the bilinear form
   $$A^i\times A^i \to \mathbb{R}, \quad (x, y) \mapsto (-1)^i \deg_A(\ell^{n - 2i} xy)$$
   is positive definite on the kernel of multiplication by $\ell^{n - 2i + 1}$. 
\end{enumerate}

For a projective simplicial toric variety $X$ over some field $k$, the Chow cohomology ring $A(X)$ with real coefficients is equipped with a degree map corresponding to the pushforward to a point and an open convex cone $\mathcal{K}_X$ in $A^1$ given by positive linear combinations of classes of ample divisors. With respect to this data, the ring $A(X)$ satisfies the K\"{a}hler package. This is a consequence of general results on the cohomology of complex projective manifolds \cite[Chapter 3]{Huybrechts}, as the Chow cohomology ring of $X$ coincides with the singular cohomology ring of the complex toric variety with the same normal fan as $X$. This fact also has a direct combinatorial proof \cite{McMullen93}. 

Let $\mathrm{M}$ be a matroid of rank $r$ on $[N] = \{1, \dotsc, N\}$. Let $\Sigma$ be a projective simplicial fan in $\mathbb{R}^m$. In order to include a prominent example (Example~\ref{ex:taut}), we allow $\Sigma$ to have a nontrivial lineality space, i.e., we consider toric varieties which are equipped with an action of a torus with a nontrivial stabilizer. 

Let $\varphi^* \colon \mathbb{Z}^N \to \mathbb{Z}^m$ be a linear map, and assume that $\Sigma$ refines the normal fan of $\varphi^*(P(\mathrm{M}))$. Equivalently, the dual map $\varphi_* \colon \mathbb{Z}^m \to \mathbb{Z}^N$ induces a map of fans from $\Sigma$ to $\Sigma_{P(\mathrm{M})}$. As described previously, for any matroid $\mathrm{M}$ of rank $r$, there are piecewise polynomial functions $\mathfrak{c}_0 = 1$, $\mathfrak{c}_1, \dotsc, \mathfrak{c}_r$ on $\Sigma_{P(\mathrm{M})}$. Pulling these back to $\Sigma$, we have induced classes $c_0 = 1, c_1, \dotsc, c_r \in A(\Sigma)$. 

\begin{example}\label{ex:taut}
Suppose that $\Sigma$ is the permutohedral fan, i.e., the fan whose maximal cones are the Weyl chambers of the type $A_{N-1}$ root system. If $\varphi$ is the identity, then $c_i$ is the class denoted $c_i(\mathcal{Q}_{\mathrm{M}^{\perp}}^{\vee})$ in \cite{BEST}. If $\varphi$ is the negation map on $\mathbb{R}^N$, then $c_i$ is the class denoted $c_i(\mathcal{S}_{\mathrm{M}})$ there. 
\end{example}

\begin{example}
In \cite{EHL}, the authors constructed classes on the stellahedral toric variety associated to matroids called \emph{augmented tautological classes}. As described in \cite[Example 8.5 and 8.6]{EHL}, the classes denoted $s(\mathcal{Q}_{\mathrm{M}})$ and $c(\mathcal{Q}_{\mathrm{M}}^{\vee})$ there can be realized in the above manner. 
\end{example}

\begin{example}
In \cite{EFLS}, the authors constructed classes on the type $B$ permutohedral toric variety associated to delta-matroids. For a broad class of delta-matroids, those with \emph{enveloping matroids}, these classes can be realized in the above manner, see \cite[Lemma 8.8 and 8.9]{EFLS}. 
\end{example}

Let $n$ denote the dimension of the projective simplicial toric variety of $\Sigma$, i.e., the dimension of $\Sigma$ minus the dimension of the lineality space of $\Sigma$. 
Let $A(\Sigma)$ denote the Chow ring of the toric variety of $\Sigma$, so there is an isomorphism $\deg_A \colon A^n(\Sigma) \to \mathbb{R}$ which satisfies $\deg_A([\operatorname{pt}]) = 1$, where $[\operatorname{pt}]$ is the class of a point. Let $\mathcal{K} \subset A^1(\Sigma)$ be the ample cone. Set
$$B = A(\Sigma)[\zeta]/(\zeta^r + c_1 \zeta^{r-1} + \dotsb + c_r).$$
If the matroid $\mathrm{M}$ is realizable over a field $k$, then the discussion above can be run in reverse to show that there is a globally generated vector bundle on the toric variety of $\Sigma$ over $k$ whose dual has Chern classes $c_1, \dotsc, c_r$. The ring $B$ is identified with the Chow ring of the projectivization of this vector bundle. 

We will identify $A(\Sigma)$ with a subring of $B$ in the natural way. Note that $B$ is naturally graded, with $\zeta \in B^1$. 
Then there is an isomorphism $\deg_B \colon B^{n+r-1} \to \mathbb{R}$ which satisfies $\deg_B([\operatorname{pt}] \zeta^{r-1}) = 1$. 

\begin{theorem}\label{thm:kahler}
The ring $B$ has the K\"{a}hler package with respect to $\deg_B$ and the interior of the cone generated by $\mathcal{K}$ and $\zeta$. 
\end{theorem}

The proof of Theorem~\ref{thm:kahler} is purely combinatorial and does not make use of algebraic geometry. See Theorem~\ref{thm:multiplebundle} and Example~\ref{thm:matroidbundle} for extensions of Theorem~\ref{thm:kahler} to the case when there are multiple projective bundles, or when $A(\Sigma)$ is replaced by the Chow ring of a matroid.

In the case of Example~\ref{ex:taut}, this result was conjectured to the first author by Hunter Spink in 2021, and it was posed as an open problem by Christopher Eur at the BIRS workshop ``Algebraic Aspects of Matroid Theory'' in 2023. In varying levels of generality, the weaker statement that the Hodge--Riemann relations hold in degree $0$ and $1$ has been proved in \cite[Section 9]{BEST}, \cite[Section 8.3]{EHL}, and \cite[Theorem 1.19]{KKS}. The previous proofs of the Hodge--Riemann relations in degree $0$ and $1$ make use of algebraic geometry.

Theorem~\ref{thm:kahler} gives a positivity property for certain classes resembling the Chern classes of a globally generated toric vector bundle. In \cite{KavehManon} and \cite{KhanMaclagan}, a combinatorial abstraction of toric vector bundles is proposed. 
In \cite[Section 6]{KavehManon}, there is a notion of global generation for this concept. It would be interesting to compare this notion with the classes that appear in Theorem~\ref{thm:kahler}.

Theorem~\ref{thm:kahler} is part of a long line of results proving the K\"{a}hler package for combinatorially defined objects, see, e.g., \cites{McMullen93,Karu2004,EliasWilliamson,AHK18,ADH}. As discussed in \cite{HuhICM1} and \cite{ELNotices}, knowing that an algebra has the K\"{a}hler package has many applications. Besides the applications to log-concavity statements that come from the Hodge--Riemann relations in degree $0$ and $1$, we use Theorem~\ref{thm:kahler} to prove some Bloch--Gieseker style statements. See Theorem~\ref{thm:blochgieseker}. Additionally, we note that Theorem~\ref{thm:kahler} implies the main result of \cite{AHK18} in a simple way, see Example~\ref{ex:AHK}. Our proof of Theorem~\ref{thm:kahler} crucially uses the main result of \cite{AHK18}, so this does not give a new proof, but it indicates the power of the statement. 

One reason why Theorem~\ref{thm:kahler} is difficult to prove is that, in a key case, the ring $B$ can be identified with the Chow ring of a projective simplicial toric variety (even for non-realizable matroids), but the cone $\mathcal{K}$ does not consist of ample divisors, see Example~\ref{ex:nonnef}. Most existing proofs of the K\"{a}hler package in combinatorial settings are based on an induction which, along the way, proves that the elements of $\mathcal{K}$ are ample. See \cite[Section 5]{ELNotices} for a description of this inductive strategy. Such a strategy cannot be used in this case, because it would prove that the elements of $\mathcal{K}$ are also ample on this toric variety.

\medskip

If the matroid $\mathrm{M}$ is realizable over a field of characteristic $0$, then $B$ is the cohomology ring of a projective bundle over the toric variety of $\Sigma$ whose ample cone contains $\mathcal{K}$. In this case Theorem~\ref{thm:kahler} follows from general results about the cohomology of complex projective manifolds. 
However, Theorem~\ref{thm:kahler} is new even if the matroid is realizable over a field of positive characteristic. We can use this to verify new cases of Grothendieck's standard conjecture of Hodge type \cite{Grothendieck}. For a smooth projective variety $X$, let $A_{\operatorname{num}}(X)$ denote the Chow ring of $X$ modulo numerical equivalence. 

\begin{conjecture}[Hdg(X)]\label{conj:Hodge}
Let $X$ be a smooth projective variety of dimension $d$ over an algebraically closed field, and let $\ell \in A^1_{\operatorname{num}}(X)$ be an ample class. Then for any $i \le d/2$, the symmetric bilinear form $(x, y) \mapsto (-1)^i\deg(\ell^{d - 2i} xy)$ is positive definite on the kernel of multiplication by $\ell^{d - 2i + 1}$. 
\end{conjecture}

If $X$ is over the complex numbers and its cohomology is generated by algebraic cycles, then Conjecture~\ref{conj:Hodge} is a consequence of classical Hodge theory. However, if $X$ is over a field of positive characteristic, then Conjecture~\ref{conj:Hodge} is not known even if its \'{e}tale cohomology is generated by algebraic cycles. 

We are interested in projectivizations of toric vector bundles over projective simplicial toric varieties. As the \'{e}tale cohomology of a smooth projective toric variety is generated by algebraic cycles, the same is true for a projective bundle over a smooth projective toric variety. However, Conjecture~\ref{conj:Hodge} is open for the projectivization of a toric vector bundle over a field of positive characteristic. Theorem~\ref{thm:kahler} implies the following result. 

\begin{corollary}\label{thm:standard}
Let $X$ be a smooth projective toric variety, and let $\mathcal{E}$ be a globally generated toric vector bundle on $X$. For any ample class $\ell$ on $X$, the ample class $\ell + c_1(\mathcal{O}(1))$ satisfies Conjecture~\ref{conj:Hodge} on the projectivization of $\mathcal{E}$. 
\end{corollary}

It seems unlikely that Corollary~\ref{thm:standard} can be proved by lifting the projectivization of $\mathcal{E}$ to characteristic $0$ in general. Because the moduli of toric vector bundles on smooth projective toric varieties satisfies Murphy's law \cite{ToricVBMurphy}, it is easy to construct toric vector bundles $\mathcal{E}$ on a toric variety over a field of positive characteristic that cannot be lifted to a toric vector bundle in characteristic $0$. However, it is not clear how to check that the projectivization of $\mathcal{E}$ cannot be lifted to characteristic $0$. 

\subsection*{Acknowledgements}
We thank Hunter Spink for many helpful conversations and for conjecturing a version of Theorem~\ref{thm:kahler} to us, and we thank Federico Ardila for explaining aspects of the Lagrangian combinatorics of matroids to us. This work was conducted while the first author was at the Institute for Advanced Study, where he is supported by the Charles Simonyi Endowment and the Oswald Veblen Fund. The second author is partially supported by NSF Grant DMS-2053288, a U.S. Department of Education GAANN award, and the Simons Foundation SFI-MPS-SDF-00015018.

\section{Fans and a key computation}\label{sec:biperm}

In this section, we recall some fundamental facts about fans and their Chow rings. We then perform a key computation in the Chow ring of the bipermutohedral fan, a fan introduced in \cite{ADH}. 

\subsection{Fans and subdivisions}

All fans will be rational and polyhedral, and we will allow fans to have nontrivial lineality spaces. A fan in $\mathbb{R}^m$ with lineality space $L$ induces a fan without lineality in $\mathbb{R}^m/L$. 
As mentioned in the introduction, we allow fans to have lineality because the normal fan of a matroid polytope $\Sigma_{P(\mathrm{M})}$ naturally has a lineality space, and the distinguished piecewise polynomial functions on it do not descend to the quotient by the lineality space. 

Given a simplicial fan $\Sigma$ in $\mathbb{R}^m$ with lineality space $L$, its \emph{Chow ring} $A(\Sigma)$ is the Chow ring of the toric variety associated to the image of $\Sigma$ in $\mathbb{R}^m/L$. We will always consider Chow rings with real coefficients. There is an explicit presentation of $A(\Sigma)$, see \cite[pg. 106]{Ful93}. The group $A^1(\Sigma)$ can be interpreted as the space of piecewise linear functions on $\Sigma$ modulo the space of globally linear functions on $\Sigma$.

A subdivision of a fan $\Sigma$ in $\mathbb{R}^m$ with lineality space $L$ is a fan $\tilde{\Sigma}$ in $\mathbb{R}^m$ that has lineality space $L$ and has the property that each cone of $\Sigma$ is a union of cones of $\tilde{\Sigma}$. This induces a subdivision of the corresponding fans in $\mathbb{R}^m/L$, which corresponds to a birational morphism from the toric variety of $\tilde{\Sigma}$ to the toric variety of $\Sigma$. If $\Sigma$ and $\tilde{\Sigma}$ are both simplicial, then this induces a pullback map $A(\Sigma) \to A(\tilde{\Sigma})$.

\begin{proposition}\label{prop:subdivision}
Let $\Sigma$ be a simplicial fan, and let $\tilde{\Sigma}$ be a simplicial subdivision. Then the pullback map $A(\Sigma) \to A(\tilde{\Sigma})$ is injective. 
\end{proposition}

\begin{proof}
This follows from a description of $A(\Sigma)$ in terms of Minkowski weights, see \cite[Proposition 5.6]{AHK18}, which are functions on the cones of $\Sigma$. A nonzero function will pull back to a nonzero function. 
\end{proof}

\medskip

The following example will play a crucial role in what follows. If $\mathrm{M}$ is a loopless matroid on $[N]$, then the \emph{Bergman fan} of $\mathrm{M}$ is a fan $\Sigma_{\mathrm{M}}$ in $\mathbb{R}^N$. For a subset $S$ of $[N]$, let $\mathbf{e}_S = \sum_{i \in S} \mathbf{e}_i$ be the indicator vector of $S$. The cones of the Bergman fan are in bijection with chains of proper nonempty flats of $\mathrm{M}$, with a chain $F_1 \subset \dotsb \subset F_s$ corresponding to the cone generated by $\mathbf{e}_{F_1}, \dotsc, \mathbf{e}_{F_s}$, and $\mathbb{R}\mathbf{e}_{[N]}$. If $\mathrm{M}$ has rank $r$, then $\Sigma_{\mathrm{M}}$ has dimension $r$ and has a $1$-dimensional lineality space. 

The Chow ring of $\Sigma_{\mathrm{M}}$ is called the Chow ring of $\mathrm{M}$. Chow rings of matroids have been extensively studied since their introduction in \cite{FY}. 

The Bergman fan of a Boolean matroid $U_{N,N}$ is the permutohedral fan appearing in Example~\ref{ex:taut}, and we will denote it by $\Sigma_N$. The Bergman fan of any loopless matroid on $\{1, \dotsc, N\}$ is a subfan of $\Sigma_N$. 

\medskip

A fan $\Sigma$ is called \emph{projective} if it is the normal fan of a polytope. 
A fan $\Sigma$ in $\mathbb{R}^m$ with lineality space $L$ is called \emph{quasi-projective} if it is a subfan of a projective fan in $\mathbb{R}^m$ with lineality space $L$. Let $\mathcal{K}(\Sigma)$ denote the image of the cone of strictly convex piecewise linear functions on $\Sigma$ in $A^1(\Sigma)$. If $\Sigma$ is quasi-projective, then $\mathcal{K}(\Sigma)$ is an open convex cone in $A^1(\Sigma)$. Each polytope whose normal fan has lineality space $L$ and has a subfan which is a coarsening of $\Sigma$ defines an element of $\mathcal{K}(\Sigma)$ via its support function. If $\Sigma$ is projective, then $\mathcal{K}(\Sigma)$ coincides with the ample cone of the toric variety associated to $\Sigma$. The definition of $\mathcal{K}(\Sigma)$ is somewhat subtle if $\Sigma$ is not complete, see \cite[Definition 5.1]{ADH}. 

Assume that $\Sigma$ is simplicial and quasi-projective, and let $d$ be the dimension of the quotient of $\Sigma$ by its lineality space $L$. Then the space of $d$-dimensional Minkowski weights supported on $\Sigma/L$ is naturally dual to $A^d(\Sigma)$, see \cite[Proposition 5.6]{AHK18}. 

We will need the theory of \emph{Lefschetz fans}, introduced in \cite{ADH}. A Lefschetz fan is a simplicial fan whose Chow ring satisfies the K\"{a}hler package, in the following sense. Assume that $\dim A^d(\Sigma) = 1$, and that there is a generator $w$ for the space of $d$-dimensional Minkowski weights on $\Sigma$ which is positive on all $d$-dimensional cones of $\Sigma$. Then $w$ defines an isomorphism $\deg_{\Sigma} \colon A^d(\Sigma) \to \mathbb{R}$. We say that $\Sigma$ is \emph{Lefschetz} if $A(\Sigma)$ has the K\"{a}hler package with respect to $\deg_{\Sigma}$ and $\mathcal{K}(\Sigma)$, and the same is true for every star fan of $\Sigma$. See \cite[Definition 1.5]{ADH}. We will need the following results. 

\begin{proposition}\label{prop:Lefschetzfans}
The following hold. 
\begin{enumerate}
\item \cite[Theorem 1.6]{ADH} If $\Sigma$ is a Lefschetz fan, then any quasi-projective simplicial fan with the same support as $\Sigma$ is a Lefschetz fan. 
\item \cite[Lemma 5.27]{ADH} A product of Lefschetz fans is Lefschetz. 
\item \cite[Theorem 8.9]{AHK18} The Bergman fan of a loopless matroid is Lefschetz. 
\item \cite[Example 5.7]{ADH} A projective simplicial fan is Lefschetz. 
\end{enumerate}
\end{proposition}

There is a positive weight on the Bergman fan of a loopless matroid $\mathrm{M}$ of rank $r$ on $[N]$ which takes value $1$ on all top-dimensional cones. This defines both a distinguished isomorphism $\deg_{\mathrm{M}} \colon A^{r-1}(\Sigma_{\mathrm{M}}) \to \mathbb{R}$ and a Minkowski weight $[\Sigma_{\mathrm{M}}] \in A^{N-r}(\Sigma_{N})$.

\subsection{The bipermutohedral fan}

Let $\mathrm{M}$ be a loopless matroid of rank $r$ on $[N]$. Recall that, in the introduction, we defined certain piecewise polynomial functions $c_1, \dotsc, c_r$ on a fan $\Sigma$. In this section, we will consider the case when $\Sigma$ is the permutohedral fan and $\varphi$ is the identity. We are then in the situation of Example~\ref{ex:taut}, so $c_i$ is the class denoted $c_i(\mathcal{Q}_{\mathrm{M}^{\perp}}^{\vee})$ in \cite{BEST}.

We now define the bipermutohedral fan $\Sigma_{N,N} \subseteq \mathbb{R}^N\times \mathbb{R}^N$ via an explicit description of its rays and cones. Our construction differs slightly from that in \cite[Section 2]{ADH} as we consider $\Sigma_{N,N}$ to be a fan with a nontrivial two-dimensional lineality space. The bipermutohedral fan $\Sigma_{N,N}$, and certain subfans $\Sigma_{N,\mathrm{M}}$ that we will shortly define, will serve as our tropical model for the projectivization of $\mathcal{Q}_{\mathrm{M}^{\perp}}^\vee$. In Section \ref{sec:computation}, we will carry out a key computation in the Chow rings of these fans (Theorem \ref{thm:identity}). 

A \emph{bisubset} $S\vert T$ is an ordered pair of subsets of $[N]$ such that $S\cup T = [N]$ and $S\cap T \neq [N]$. We allow for one of $S$ or $T$ to be empty, and we say that $S \vert T$ is \emph{proper} if both $S$ and $T$ are nonempty.
Given two bisubsets $S\vert T$ and $S'\vert T'$, we say that $S\vert T \leq S' \vert T'$ if $S\subseteq S'$ and $T\supseteq T'$. Note that $\emptyset \vert [N] \leq S\vert T \leq [N] \vert \emptyset$ for all bisubsets $S\vert T$. Let $\mathcal{S}\vert \mathcal{T}= (S_1\vert T_1 < S_2 \vert T_2 < \dotsb < S_\ell \vert T_\ell)$ be a chain of proper bisubsets. An integer $0\leq j\leq \ell$ is a \emph{gap index} of $\mathcal{S}\vert \mathcal{T}$ if, after setting $S_0 \vert T_0 = \emptyset \vert [N]$ and $S_{\ell+1} \vert T_{\ell+1} = [N] \vert \emptyset$, $S_j \cup T_{j+1} \neq [N]$. We say that a chain of proper bisubsets $\mathcal{S}\vert \mathcal{T}$ is a \emph{biflag} if it has a gap index. If $\mathcal{S}\vert \mathcal{T}$ is a biflag of length $\ell$, we will set $S_0\vert T_0 = \emptyset \vert [N]$ and $S_{\ell+1} \vert T_{\ell+1} = [N] \vert \emptyset$ for convenience.

For an ordered pair $S\vert T$ of subsets of $[N]$, let $\mathbf{e}_{S\vert T} = (\sum_{i\in S} \mathbf{e}_i, \sum_{j\in T} \mathbf{e}_j) \in \mathbb{R}^N \times \mathbb{R}^N$ be the indicator vector of $S\vert T$.

\begin{definition}
    The bipermutohedral fan $\Sigma_{N,N}$ is the complete simplicial fan in $\mathbb{R}^N \times \mathbb{R}^N$ with the cones
    \[\sigma_{\mathcal{S}\vert \mathcal{T}} = \operatorname{span}\{\mathbf{e}_{\emptyset \vert [N]}, \mathbf{e}_{[N]\vert \emptyset}\} + \operatorname{cone}\{\mathbf{e}_{S\vert T} \}_{S\vert T \in \mathcal{S} \vert \mathcal{T}}, \quad \text{where $\mathcal{S}\vert \mathcal{T}$ is a biflag of proper bisubsets.} \]
\end{definition}

We say that a bisubset $S\vert F$ is a \emph{biflat} of $\mathrm{M}$ if $F$ is a flat of $\mathrm{M}$. Note that this definition differs from the one in \cite{ADHCombo, ADH}. If a biflag $\mathcal{S}\vert \mathcal{F}$ consists entirely of biflats of $\mathrm{M}$, then we say that $\mathcal{S}\vert \mathcal{F}$ is a biflag of $\mathrm{M}$. An index $j$ is a gap index of $\mathcal{S}\vert \mathcal{F}$ if and only if $\overline{S_j^c} \not\subseteq F_{j+1}$ where $\overline{S_j^c}$ denotes the inclusion-minimal flat of $\mathrm{M}$ containing $S_j^c$.

The \emph{projective bundle fan} of $\mathrm{M}$, denoted by $\Sigma_{N,\mathrm{M}}$, is the subfan of $\Sigma_{N,N}$ whose support is $\mathbb{R}^N\times \vert \Sigma_\mathrm{M}\vert $. More concretely, $\Sigma_{N,\mathrm{M}}$ is the subfan of $\Sigma_{N,N}$ whose cones correspond to the biflags of $\mathrm{M}$. The projective bundle fan of the Boolean matroid $U_{N,N}$ is the bipermutohedral fan $\Sigma_{N,N}$.

For $j\in [N]$, let $\alpha_j$ be the piecewise linear function on $\mathbb{R}^N$ defined by $\alpha_j(z) = \min_{i\in [N]}(z_j-z_i)$. Each $\alpha_j$ is a piecewise linear function on $\Sigma_N$ and, modulo global linear functions, they all belong to the same equivalence class which we denote by $\alpha$. Let $\gamma_j$ and $\overline{\gamma}_j$ be the pullback of $\alpha_j$ along the first and second projections $\pi_1, \pi_2 \colon \mathbb{R}^N\times \mathbb{R}^N \to \mathbb{R}^N$, respectively. Finally, let $\delta_j$ be the piecewise linear function on $\mathbb{R}^N\times \mathbb{R}^N$ obtained by pulling back $\alpha_j$ along the addition map 
\[\mathbb{R}^N\times \mathbb{R}^N \to \mathbb{R}^N, \quad (z,w)\mapsto z+w. \]
In \cite[Section 2]{ADH} it is shown that the functions $\gamma_j$, $\overline{\gamma}_j$ and $\delta_j$ are piecewise linear functions on the bipermutohedral fan and hence restrict to piecewise linear functions on $\Sigma_{N,\mathrm{M}}$ for every matroid $\mathrm{M}$. Modulo global linear functions, the equivalence classes of $\gamma_j$, $\overline{\gamma}_j$ and $\delta_j$ do not depend on the choice of $j\in [N]$, and we write $\gamma$, $\overline{\gamma}$ and $\delta$ for their common equivalence classes.

In what follows, it will be convenient to have a presentation of $A(\Sigma_{N,\mathrm{M}})$ in terms of generators and relations. First, consider the polynomial ring $S_{N,\mathrm{M}}$ with real coefficients and variables $x_{S\vert F}$ indexed by the proper biflats of $\mathrm{M}$. For any set of biflats $\mathcal{S}\vert \mathcal{F}$, let $x_{\mathcal{S}\vert \mathcal{F}}$ be the monomial
\[x_{\mathcal{S}\vert \mathcal{F}} = \prod_{S\vert F \in \widehat{\mathcal{S}\vert \mathcal{F}}}x_{S\vert F}\]
where $\widehat{\mathcal{S}\vert \mathcal{F}} \subseteq \mathcal{S}\vert\mathcal{F}$ is the set of proper biflats of $\mathcal{S}\vert\mathcal{F}$.
For every $j\in [N]$, consider the following elements of $S_{N, \mathrm{M}}$:
\[\gamma_j = \sum_{j \in S, \,  S\not= [N]} x_{S|F}, \quad \overline{\gamma}_j = \sum_{j \in F, \, F\not= [N]} x_{S|F}, \text{ and } \delta_j = \gamma_j + \overline{\gamma}_j - \sum_{S,F \not=[N]} x_{S|F}.\]
We write $I_{N,\mathrm{M}}$ for the ideal generated by the monomials $x_{\mathcal{S}\vert \mathcal{F}}$, where $\mathcal{S}\vert \mathcal{F}$ is not a biflag, and $J_{N,\mathrm{M}}$ for the ideal generated by the linear forms $\gamma_i-\gamma_j$ and $\overline{\gamma}_i-\overline{\gamma}_j$ for each $i$ and $j$ in $[N]$. 

\begin{definition}
    The Chow ring of the projective bundle fan of $\mathrm{M}$ is the quotient
    \[A(\Sigma_{N,\mathrm{M}}) = S_{N,\mathrm{M}}/(I_{N,\mathrm{M}}+J_{N,\mathrm{M}}). \]
\end{definition}

Let $\gamma$,  $\overline{\gamma}$, and $\delta$ be the classes of the linear forms $\gamma_j$, $\overline{\gamma}_j$ and $\delta_j$ in $A(\Sigma_{N,\mathrm{M}})$. After identifying $A^1(\Sigma_{N,\mathrm{M}})$ with the space of piecewise linear functions on $\Sigma_{N,\mathrm{M}}$ modulo global linear functions, these classes agree with the previously defined classes by the same name. 

For $1 \le j \le r$, we consider the class in $A^1(\Sigma_{N, N})$ given by
\begin{equation*}
u_j = \sum_{S \not= [N], \, \operatorname{rk}(S^c) < j} x_{S|T} - \gamma.
\end{equation*}

\begin{lemma}\label{lem:uReduction}
The $i$th elementary symmetric function of $\{u_1, \dotsc, u_r\}$ is equal to $\pi_1^* c_i$. 
\end{lemma}

\begin{proof}
In \cite[Appendix III]{BEST}, the authors describe classes $w_1, \dotsc, w_r$ in $A^1(\Sigma_N)$ such that $c_i$ is the $i$th elementary symmetric function of $\{w_1, \dotsc, w_r\}$. These classes depend on the choice of some auxiliary matroids; using the ``full Higgs lift'' as described in \cite[Appendix III]{BEST}, we have
$$w_i =  \sum_{\operatorname{rk}(S^c) < i, \, S\not= [N]} x_S - \alpha.$$
By definition, $\pi_1^* \alpha = \gamma$. We have $\pi_1^* x_S = \sum_{T} x_{S|T}$, giving the result. 
\end{proof}

\subsection{A projective bundle relation}\label{sec:computation}

For this section, we fix a loopless matroid $\mathrm{M}$ of rank $r$ on the ground set $[N]$. We will often refer to the biflats and biflags of $\mathrm{M}$ simply as biflats and biflags. We will also often restrict classes like $\gamma, \overline{\gamma}, c_1, \dotsc, c_r,$ and $\delta$ to $A(\Sigma_{N, \mathrm{M}})$ without comment. 
The goal of this section is to prove that the following relation holds in the Chow ring of $\Sigma_{N,\mathrm{M}}$.

\begin{theorem}\label{thm:identity}
In $A(\Sigma_{N, \mathrm{M}})$, we have $\delta^r + c_1 \delta^{r-1} + \dotsb + c_r = 0$. 
\end{theorem}

The proof of Theorem~\ref{thm:identity} is quite involved and requires us to start with some structural definitions.

Recall that for a biflag $\mathcal{U}\vert \mathcal{H}$ of length $k$, an integer $0\leq j\leq k$ is a gap index of $\mathcal{U}\vert \mathcal{H}$ if $U_j \cup H_{j+1}$ is not equal to the entire ground set of $\mathrm{M}$. By definition, all biflags have at least one gap index and we will pay special attention to their first gap index. To emphasize this, we will write biflags of $\mathrm{M}$ as 
\[\mathcal{S}\vert\mathcal{F}<\mathcal{T}\vert\mathcal{G} = (S_1 \vert F_1 < \dotsb < S_s\vert F_s < T_1\vert  G_1<\dotsb <T_\ell \vert G_\ell) \]
where $s$ is the smallest gap index of $\mathcal{S}\vert\mathcal{F}<\mathcal{T}\vert\mathcal{G}$. We refer to $\mathcal{S}\vert \mathcal{F}$ and $\mathcal{T}\vert \mathcal{G}$ as the first and second components of $\mathcal{S}\vert\mathcal{F}<\mathcal{T}\vert\mathcal{G}$, respectively. Note that it is possible for a biflag to have an empty first or second component. We will use $s$ to refer to the length of $\mathcal{S}\vert \mathcal{F}$ and $\ell$ for the length of $\mathcal{T}\vert \mathcal{G}$. Finally, it will be convenient for us to define $S_0\vert F_0= \emptyset \vert [N]$, $T_0\vert G_0 = S_s \vert F_s$, and $T_{\ell+1} \vert G_{\ell+1} = [N]\vert \emptyset$.

\begin{example}\label{ex:biflag}
We will use the graphic matroid of the pyramid $\Gamma$, illustrated in Figure~\ref{fig:graph}, as a running example in this section. Set $E = [8]$. The set of gap indices of the biflag
\[126\vert E < 126\vert 34578  < 1246 \vert 34578 < 12456\vert 378 < 124567\vert 378\]
is equal to $\{3,5\}$. As $3$ is the first gap index, we write this biflag as $\mathcal{S}\vert \mathcal{F} < \mathcal{T}\vert \mathcal{G}$ where 
\[\mathcal{S}\vert \mathcal{F} = (126\vert E < 126\vert 34578  < 1246 \vert 34578) \quad\text{and}\quad \mathcal{T}\vert \mathcal{G} = (12456\vert 378 < 124567\vert 378).\]

\begin{figure}[!ht]
    \begin{tikzpicture} [scale=1]
\node[fill, circle, minimum size =0.5em, inner sep =0em] (e) at (0,0) {};
\node[fill, circle, minimum size =0.5em, inner sep =0em] (a) at (1,1) {};
\node[fill, circle, minimum size =0.5em, inner sep =0em] (b) at (-1,1) {};
\node[fill, circle, minimum size =0.5em, inner sep =0em] (c) at (-1,-1) {};
\node[fill, circle, minimum size =0.5em, inner sep =0em] (d) at (1,-1){};

\draw (a) to node[above]{$1$} (b);
\draw (b) to node[left]{$2$} (c);
\draw (c) to node[below]{$3$} (d);
\draw (d) to node[right]{$4$} (a);

\draw (a) to node[left]{$5$} (e);
\draw (b) to node[right]{$6$} (e);
\draw (c) to node[right]{$7$} (e);
\draw (d) to node[left]{$8$} (e);
\end{tikzpicture}
\caption{The pyramid graph $\Gamma$}
\label{fig:graph}
\end{figure}
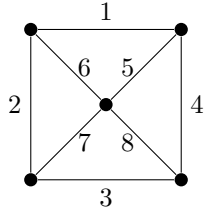
\end{example}

The following kinds of biflags have second components with desirable behavior.
\begin{definition}
  For a biflag $\mathcal{S}\vert \mathcal{F} < \mathcal{T}\vert \mathcal{G}$ and an index $0\leq i \leq \ell$, we say that $\mathcal{S}\vert \mathcal{F} < \mathcal{T}\vert \mathcal{G}$ is \emph{lexicographically decreasing at $i$} if
  \[\min\left(\overline{S_s^c} \setminus G_{i+1}\right) \in G_i. \]
  If $\mathcal{S}\vert \mathcal{F} < \mathcal{T}\vert \mathcal{G}$ is lexicographically decreasing at all such indices, then we say that $\mathcal{S}\vert \mathcal{F} < \mathcal{T}\vert \mathcal{G}$ is \emph{lexicographically decreasing}. 
  \end{definition}
As $S_s\vert F_s$ is the start of the first gap of $\mathcal{S}\vert \mathcal{F} < \mathcal{T}\vert \mathcal{G}$, all biflags are lexicographically decreasing at index $0$.
See Example~\ref{ex:lex_dec} for an illustration of this definition. We justify the name ``lexicographically decreasing'' in Remark \ref{rmk:lex_dec}. A biflag consisting of a single proper biflat $T\vert G$ is lexicographically decreasing whenever $\min([N])=1\in G$. The following lemma gives us strong control over the ranks of $\overline{S_s^c}\cap G_i$ and $T_i^c$ in lexicographically decreasing biflags.

\begin{lemma}\label{lem:dyck}
  Suppose that $\mathcal{S}\vert \mathcal{F} < \mathcal{T}\vert \mathcal{G}$ is lexicographically decreasing, and let $\operatorname{rk}(\overline{S_s^c})=a$. Then
  \[\operatorname{rk}(\overline{S_s^c} \cap G_i)> \operatorname{rk}(\overline{S_s^c} \cap G_{i+1}) \quad \text{and}\quad \operatorname{rk}(T_i^c) \leq \operatorname{rk}(\overline{S_s^c} \cap G_i) \leq a-i \]
  for all $1\leq i \leq \ell$. 
\end{lemma}
\begin{proof}
 As $G_{i+1} \cup \min(\overline{S_s^c} \setminus G_{i+1}) \subseteq G_i$, it follows that $\overline{S_s^c} \cap G_{i+1} \subsetneq \overline{S_s^c} \cap G_i$. The sets $\overline{S_s^c} \cap G_{i+1}$ and $\overline{S_s^c} \cap G_i$ are flats so the first inequality now follows. The second inequality follows from the fact that $T_i^c \subseteq S_s^c \cap G_i$. 

  We now prove the third inequality. For the sake of contradiction, suppose there is some index $1\leq i\leq \ell$ where $\operatorname{rk}(\overline{S_s^c} \cap G_i)>a-i$. By the first inequality,  this implies that $\operatorname{rk}(\overline{S_s^c} \cap G_1 ) > a-1$. As the rank of $\overline{S_s^c}$ is $a$, we have that $\overline{S_s^c}\cap G_1 = \overline{S_s^c}$, so $\overline{S_s^c} \subseteq G_1$. This contradicts the fact that $s$ is a gap index of $\mathcal{S}\vert \mathcal{F} < \mathcal{T}\vert \mathcal{G}$. 
\end{proof}

\begin{remark}\label{rmk:lex_dec}
  Let $\mathcal{S}\vert \mathcal{F} < \mathcal{T}\vert \mathcal{G}$ be a lexicographically decreasing biflag whose second component has length equal to $\operatorname{rk}(\overline{S_s^c})-1=a-1$. Lemma \ref{lem:dyck} tells us that
  \[\overline{S_s^c} \cap \mathcal G^{\text{op}} = \left( \emptyset = \overline{S_s^c} \cap G_{a} \subset \overline{S_s^c} \cap G_{a-1} \subset \overline{S_s^c} \cap G_{a-2} \subset \dotsb \subset \overline{S_s^c} \cap G_1 \subset \overline{S_s^c}\right)\]
  is a saturated flag of flats.  In the interval $[\emptyset, \overline{S_s^c}]$ of the lattice of flats of $\mathrm{M}$, $\overline{S_s^c} \cap \mathcal G^{\text{op}}$ is the unique maximal flag of flats whose EL-labeling, in the sense of \cite{BjornerShellable}, is lexicographically increasing. Because $\overline{S_s^c} \cap \mathcal{G}^{\text{op}}$ is read in the opposite order as $\mathcal{T}\vert \mathcal{G}$, we've chosen to call the biflags $\mathcal{S} \vert \mathcal{F} < \mathcal{T}\vert \mathcal{G}$ ``lexicographically decreasing''.
\end{remark}

\begin{example}\label{ex:lex_dec}
    The biflag $\mathcal{S} \vert \mathcal{F} < \mathcal{T}\vert \mathcal{G}$ of Example~\ref{ex:biflag} is \emph{not} lexicographically decreasing. In particular, it fails to be lexicographically decreasing at index $1$. To see this, observe that $\overline{S_3^c}= 34578$  and $G_1=G_2=378$ so 
$\min(\overline{S_3^c}\setminus G_2) = 4 \not \in G_1$.

    In contrast, for the pyramid graph $\Gamma$ of Figure~\ref{fig:graph}, the biflag $\mathcal{S}' \vert \mathcal{F}' < \mathcal{T}'\vert \mathcal{G}'$ given by
    \[\mathcal{S}'\vert \mathcal{F}' = (1267\vert E) \quad\text{and}\quad \mathcal{T}'\vert \mathcal{G}' = (125678\vert 1234 < E\vert 3)\]
    is lexicographically decreasing. In agreement with Remark~\ref{rmk:lex_dec}, we see that $\overline{S_1'^c} \cap G_1' = 34$ and $\overline{S_1'^c}\cap G_2'=3$. Note that this condition does not determine $G_1'$ or $G_2'$. For example, replacing $G_1'$ with the flat $34$ results in a different lexicographically decreasing biflag.
\end{example}

It follows from the definitions that, in $A^1(\Sigma_{N, N})$, we have
$$\delta + u_j = \overline{\gamma}  + \sum_{S \not= [N], \, \operatorname{rk}(S^c) < j} x_{S|[N]} - \sum_{S, T \not= [N], \, \operatorname{rk}(S^c) \ge j} x_{S|T}.$$
Set $v_j^+ = \sum_{S \not= [N], \, \operatorname{rk}(S^c) < j} x_{S|[N]}$ and $v_j^- = \sum_{S, T \not= [N], \, \operatorname{rk}(S^c) \ge j} x_{S|T}$, so $\delta + u_j = \overline{\gamma} + v_j^+ - v_j^-$. 
We will now work towards proving the following proposition.
\begin{proposition}\label{thrm:min_dec}
  Let $\mathcal{S}\vert \mathcal{F} =  ( S_1|F_1 < \dotsb< S_s\vert F_s)$ be a chain of proper biflats with no gap indices strictly smaller than $s$, and let $a=\operatorname{rk}(\overline{S_s^c})$. 
   For any number $0\leq \ell \leq a$, we have that
  \[\sum_{\# \mathcal{T}\vert \mathcal{G} = \ell} x_{\mathcal{S}\vert \mathcal{F} < \mathcal{T}\vert \mathcal{G}} \prod_{i=1}^{a-\ell} \left(\overline\gamma-v_i^-\right) =0\]
    in $A(\Sigma_{N,\mathrm{M}})$ where the sum is over all lexicographically decreasing biflags whose first component is $\mathcal{S}\vert \mathcal{F}$ and whose second component has length $\ell$.
  \end{proposition}

  The $\ell=0$ case of Proposition~\ref{thrm:min_dec} is an important part of the proof of Theorem~\ref{thm:identity}. Namely, it will let us prove Proposition~\ref{prop:reductionToVs}. 
  Our proof of Proposition~\ref{thrm:min_dec} is inductive and relies on an explicit understanding of the product $\sum x_{\mathcal{S}\vert \mathcal{F} < \mathcal{T}\vert \mathcal{G}} (\overline\gamma -v_{a-\ell}^-)$ in $A(\Sigma_{N,\mathrm{M}})$. Similar to the intricate computations of \cite{ADHCombo}, the fundamental hurdle is in understanding which choice of representative $\overline\gamma_i= \overline\gamma$  to use at each step.

  \begin{definition}\label{def:canonical}
    Suppose $\mathcal{S}\vert \mathcal{F}< \mathcal{T}\vert \mathcal{G}$ is a lexicographically decreasing biflag whose first and second components have lengths $s$ and $\ell$, respectively. Let $\operatorname{rk}(S_s^c)=a$, and let $1\leq i \leq \ell+1$ be the smallest index such that $\operatorname{rk}(T_i^c)<a-\ell$. Finally, let
    \[e = \min(\overline{S_s^c}\setminus G_i).\]
    The \emph{canonical expansion} of $x_{\mathcal{S}\vert \mathcal{F}< \mathcal{T}\vert \mathcal{G}}(\overline\gamma -v_{a-\ell}^-)$ is the expression
      \[x_{\mathcal{S}\vert \mathcal{F}< \mathcal{T}\vert \mathcal{G}}(\overline\gamma -v_{a-\ell}^-) = x_{\mathcal{S}\vert \mathcal{F}< \mathcal{T}\vert \mathcal{G}}(\overline\gamma_e -v_{a-\ell}^-). \]
    \end{definition}
Note that we allow the index $i=\ell+1$ in Definition \ref{def:canonical}. As $T_{\ell+1}^c=\emptyset$, this ensures that the canonical expansion is well defined. Using the relations of $A(\Sigma_{N,\mathrm{M}})$, we can rewrite the canonical expansion of $x_{\mathcal{S}\vert \mathcal{F}< \mathcal{T}\vert \mathcal{G}}(\overline\gamma -v_{a-\ell}^-)$ as a sum of squarefree monomials whose coefficients are $0$, $1$ or $-1$.
    
    \begin{lemma}\label{lem:canonical}
      In $A(\Sigma_{N,\mathrm{M}})$, the canonical expansion is equal to
\begin{equation}\label{eq:canonical}
        x_{\mathcal{S}\vert \mathcal{F}< \mathcal{T}\vert \mathcal{G}}(\overline\gamma_e -v_{a-\ell}^-)  = \sum_{\substack{\operatorname{rk}(U^c)<a-\ell,\\ e\in H \neq [N]}}x_{\mathcal{S}\vert \mathcal{F}< \mathcal{T}\vert \mathcal{G}}x_{U\vert H}
    - \sum_{\substack{\operatorname{rk}(U^c)\geq a-\ell,\\ e\not\in H}}x_{\mathcal{S}\vert \mathcal{F}< \mathcal{T}\vert \mathcal{G}}x_{U\vert H}\,.
  \end{equation}
  
In both sums, the nonzero terms correspond to biflats $U\vert H$ with 
$T_{i-1}\vert G_{i-1} < U\vert H < T_i\vert G_i$.
    \end{lemma}
    \begin{proof}
      Equation~(\ref{eq:canonical}) follows directly from the definitions. We now check the second claim. First, let $U\vert H$ correspond to a nonzero term of the first sum. As $\operatorname{rk}(U^c)<a-\ell$, we have that $T_{i-1} \subsetneq U$, so $T_{i-1}\vert G_{i-1} < U\vert H$. As $e\in H$ and $e\not\in G_i$, we have that $H\supsetneq G_i$, so $U\vert H < T_i\vert G_i$. 

      To see that the claim holds for the second sum, suppose that $U\vert H$ corresponds to one of its nonzero terms. As ${\mathcal{S}\vert \mathcal{F}< \mathcal{T}\vert \mathcal{G}}$ is lexicographically decreasing, $e\in G_{i-1}$. This implies that $G_{i-1}\supsetneq H$, so  $T_{i-1}\vert G_{i-1} < U\vert H$. As $\operatorname{rk}(U^c)\geq a-\ell$, we  have $U\subsetneq T_i$ and so $U\vert H < T_i\vert G_i$.      
    \end{proof}

    \begin{example}\label{ex:canonical}
    For the pyramid graph $\Gamma$ of Figure~\ref{fig:graph}, let $\mathcal{S}\vert \mathcal{F} < \mathcal{T} \vert \mathcal{G}$ be the length two biflag given by $\mathcal{S}\vert \mathcal{F} = (1267\vert E)$ and $\mathcal{T}\vert \mathcal{G} = (E\vert 3)$. Following Definition~\ref{def:canonical}, we have that $e= \min(\overline{S_1^c}\setminus G_1)= \min(34578\setminus 3)=4$ and the canonical expansion of $x_{\mathcal{S}\vert \mathcal{F}< \mathcal{T}\vert \mathcal{G}}(\overline\gamma -v_{2}^-)$ is $x_{\mathcal{S}\vert \mathcal{F}< \mathcal{T}\vert \mathcal{G}}(\overline\gamma -v_{2}^-)=x_{\mathcal{S}\vert \mathcal{F}< \mathcal{T}\vert \mathcal{G}}(\overline\gamma_4 -v_{2}^-)$. The biflags
    \[1267\vert E < E \vert 34 < E \vert 3 \quad\text{and}\quad 1267\vert E < 124567\vert 378 < E\vert 3 \]
    correspond to positive and negative terms of Equation~(\ref{eq:canonical}), respectively.
    \end{example}

    Given a lexicographically decreasing biflag $\mathcal{S}\vert \mathcal{F}< \mathcal{T}\vert \mathcal{G}$, let $\operatorname{pos}(\mathcal{S}\vert \mathcal{F}< \mathcal{T}\vert \mathcal{G})$ be the set of biflags that correspond to the nonzero monomials in the positive sum of Equation~(\ref{eq:canonical}). Similarly, let $\operatorname{neg}(\mathcal{S}\vert \mathcal{F}< \mathcal{T}\vert \mathcal{G})$ be the set of biflags that correspond to the nonzero monomials in the negative sum of Equation~(\ref{eq:canonical}).

    \begin{lemma}\label{lem:disjoint}
      Let $\mathcal{S}\vert \mathcal{F}< \mathcal{T}\vert \mathcal{G}$ and $\mathcal{S}\vert \mathcal{F}< \mathcal{T}'\vert \mathcal{G}'$ be different lexicographically decreasing biflags whose second components have length $\ell$ and whose first components are equal. The set $\operatorname{pos}(\mathcal{S}\vert \mathcal{F}< \mathcal{T}\vert \mathcal{G})$ is disjoint from $\operatorname{pos}(\mathcal{S}\vert \mathcal{F}< \mathcal{T}'\vert \mathcal{G}')$ and the set $\operatorname{neg}(\mathcal{S}\vert \mathcal{F}< \mathcal{T}\vert \mathcal{G})$ is disjoint from $\operatorname{neg}(\mathcal{S}\vert \mathcal{F}< \mathcal{T}'\vert \mathcal{G}')$.
    \end{lemma}
    \begin{proof}
      Let $i$ and $i'$ be the first indices such that $\operatorname{rk}(T_i^c)<a-\ell$ and $\operatorname{rk}(T_{i'}'^c)<a-\ell$. Lemma \ref{lem:canonical} ensures that our claims hold if $i\neq i'$. If $i=i'$, then there is some index $j$ such that $T_j\vert G_j \neq T_j'\vert G_j'$. If $j<i$, then any biflags of $\operatorname{pos}(\mathcal{S}\vert \mathcal{F}< \mathcal{T}\vert \mathcal{G})$ and $\operatorname{pos}(\mathcal{S}\vert \mathcal{F}< \mathcal{T}'\vert \mathcal{G}')$ will still differ from each other at index $j$. If $j\geq i$, then the same is true but now at index $j+1$. This argument also shows that $\operatorname{neg}(\mathcal{S}\vert \mathcal{F}< \mathcal{T}\vert \mathcal{G})$ is disjoint from $\operatorname{neg}(\mathcal{S}\vert \mathcal{F}< \mathcal{T}'\vert \mathcal{G}')$.
\end{proof}

Let $\mathbf{A}$ denote the set of all lexicographically decreasing biflags whose first component is $\mathcal{S}\vert \mathcal{F}$ and whose second component has length $\ell$. We now consider the result of applying the canonical expansion to every term in the sum
\[\sum_{{\mathcal{S}\vert \mathcal{F} < \mathcal{T}\vert \mathcal{G}}\in \mathbf{A}} x_{\mathcal{S}\vert \mathcal{F} < \mathcal{T}\vert \mathcal{G}} (\overline\gamma -v_{a-\ell}^-)\,. \]
Define
\begin{gather*}
\operatorname{pos}(\mathbf{A}) = \bigcup_{\mathcal{S}\vert \mathcal{F} < \mathcal{T}\vert \mathcal{G}\in \mathbf{A}} \operatorname{pos}(\mathcal{S}\vert \mathcal{F} < \mathcal{T}\vert \mathcal{G}),\quad \text{and} \quad
\operatorname{neg}(\mathbf{A}) = \bigcup_{\mathcal{S}\vert \mathcal{F} < \mathcal{T}\vert \mathcal{G}\in \mathbf{A}} \operatorname{neg}(\mathcal{S}\vert \mathcal{F} < \mathcal{T}\vert \mathcal{G}).
\end{gather*}

Lemma \ref{lem:disjoint} ensures that each biflag of $\operatorname{pos}(\mathbf{A})$ can be written uniquely as a union $\mathcal{S}\vert \mathcal{F} < \mathcal{T}\vert \mathcal{G} \cup \{U\vert H\}$, where $\mathcal{S}\vert \mathcal{F} < \mathcal{T}\vert \mathcal{G}$ is a lexicographically decreasing biflag in $\mathbf{A}$. The same holds true for biflags in $\operatorname{neg}(\mathbf{A})$. Lemma \ref{lem:disjoint} also tells us that
  \begin{align}
    \sum_{\mathcal{S}\vert \mathcal{F} < \mathcal{T}\vert \mathcal{G}\in \mathbf{A}} x_{\mathcal{S}\vert \mathcal{F} < \mathcal{T}\vert \mathcal{G}} (\overline\gamma -v_{a-\ell}^-) &=\sum_{\mathcal{S}\vert \mathcal{F} < \mathcal{T}\vert \mathcal{G}\cup \{U\vert H\}\in \operatorname{pos}(\mathbf{A})} x_{\mathcal{S}\vert \mathcal{F} < \mathcal{T}\vert \mathcal{G}\cup \{U\vert H\}} \label{eq:pos}\\
    &-\sum_{{\mathcal{S}\vert \mathcal{F} < \mathcal{T}\vert \mathcal{G}\cup \{U\vert H\}}\in \operatorname{neg}(\mathbf{A})} x_{\mathcal{S}\vert \mathcal{F} < \mathcal{T}\vert \mathcal{G}\cup \{U\vert H\}} \label{eq:neg}
  \end{align}
  This equality is already a significant reduction in complexity. For different choices of representatives $\overline{\gamma}_i=\overline\gamma$, we are often left with an expression that is not the sum of squarefree monomials and whose coefficients are not equal to $0$, $1$ or $-1$. We'll now analyze the sets $\operatorname{pos}(\mathbf{A})$ and $\operatorname{neg}(\mathbf{A})$. The result of this analysis is Proposition \ref{prop:summary}, which cancels the negative terms of Equation~(\ref{eq:neg}) with a subset of the positive terms on the right hand side of Equation~(\ref{eq:pos}).

  For an integer $1\leq j \leq \ell+1$, let $\mathbf{A}_j$ be the subset of $\mathbf{A}$ consisting of the biflags $\mathcal{S}\vert \mathcal{F} < \mathcal{T}\vert \mathcal{G}$ where $j$ is the first index such that $\operatorname{rk}(T_j^c)< a-\ell$. As $\operatorname{rk}(T_{\ell+1}^c)=0$, the collection $\{\mathbf{A}_j:1\leq j \leq \ell+1\}$ partitions $\mathbf{A}$. Define $\operatorname{pos}(\mathbf{A}_j)$ and $\operatorname{neg}(\mathbf{A}_j)$ in a similar manner to $\operatorname{pos}(\mathbf{A})$ and $\operatorname{neg}(\mathbf{A})$. We now prove two technical lemmas to show that $\operatorname{neg}(\mathbf{A}) \subseteq \operatorname{pos}(\mathbf{A})$ and  describe the complement $ \operatorname{pos}(\mathbf{A}) \setminus \operatorname{neg}(\mathbf{A})$. See Figure~\ref{fig:in_total} for a diagrammatic summary.

  \begin{lemma}\label{lem:empty}
    The set $\operatorname{neg}(\mathbf{A}_{\ell+1})$ is empty.
  \end{lemma}
  \begin{proof}
 Suppose $\mathcal{S}\vert\mathcal{F} <\mathcal{T}\vert \mathcal{G}< U\vert H$ is a biflag of $\operatorname{neg}(\mathbf{A}_{\ell+1})$. Lemma \ref{lem:dyck} implies that $\operatorname{rk}(T_\ell^c)=\operatorname{rk}(\overline{S_s^c}\cap G_\ell)= a-\ell$ and $\overline{T_\ell^c} = \overline{S_s^c} \cap G_\ell$. Lemma \ref{lem:canonical} tells us that the rank of $U^c$ is at least $a-\ell$, so we also have that $\overline{U^c}=\overline{T_\ell^c}$. Again by Lemma \ref{lem:canonical}, we know that $\min(\overline{S_s^c})\not\in H$. Because $\mathcal{S}\vert\mathcal{F} <\mathcal{T}\vert \mathcal{G}$ is lexicographically decreasing, $\min(\overline{S_s^c})\in G_\ell$. Thus $\overline{S_s^c}\cap H \subsetneq \overline{S_s^c} \cap G_\ell$. Putting this all together, we get
    \[\overline{S_s^c}\cap H \subsetneq \overline{S_s^c}\cap G_\ell = \overline{T_\ell^c} = \overline{U^c}. \]
But $\overline{U^c}\subseteq \overline{S_s^c}$ so this implies that $\overline{U^c}\not\subseteq H$, which is a contradiction.
  \end{proof}
  \begin{lemma}\label{lem:neg_containment}
     For any $1\leq j \leq \ell$, $\operatorname{neg}(\mathbf{A}_{j})\subseteq \operatorname{pos}(\mathbf{A}_{j+1})$. Furthermore, the difference $\operatorname{pos}(\mathbf{A}_{j+1})\setminus \operatorname{neg}(\mathbf{A}_{j})$ is equal to the set of all lexicographically decreasing biflags
    \[\mathcal{S}\vert \mathcal{F} < T_1 \vert G_1< \dotsb <T_{j}\vert G_{j} < U\vert H < T_{j+1} \vert G_{j+1} <\dotsb < T_{\ell}\vert G_\ell\]
    whose second component has length $\ell+1$ and which obey the condition that $U\vert H$ is the first biflat such that $\operatorname{rk}(U^c)< a-\ell$.
  \end{lemma}
  \begin{proof}
By Lemma \ref{lem:canonical}, $\operatorname{pos}(\mathbf{A}_{j+1})$ consists of all of the biflags 
\[\mathcal{S}\vert \mathcal{F} < T_1 \vert G_1< \dotsb <T_{j}\vert G_{j} < U\vert H < T_{j+1} \vert G_{j+1} <\dotsb < T_{\ell}\vert G_\ell\]
that are lexicographically decreasing at indices $i \in [\ell+1]\setminus j$ and obey the condition that $U\vert H$ is the first biflat such that $\operatorname{rk}(U^c)<a-\ell$. We write these biflags as $\mathcal{S}\vert \mathcal{F}< \mathcal{T}\vert\mathcal{G} \cup \{U\vert H\}$. We first show that $\operatorname{neg}(\mathbf{A}_j)$ contains the subset of biflags $\mathcal{S}\vert \mathcal{F}< \mathcal{T}\vert\mathcal{G} \cup \{U\vert H\}$ that are not lexicographically decreasing at index $j$. Suppose that $\mathcal{S}\vert \mathcal{F}< \mathcal{T}\vert\mathcal{G} \cup \{U\vert H\}$
is not lexicographically decreasing at $j$. We claim that $\mathcal{S}\vert \mathcal{F}< \mathcal{T}\vert\mathcal{G} \cup \{U\vert H\}$ is in the set 
\[\operatorname{neg}(\mathcal{S}\vert \mathcal{F} < T_1 \vert G_1< \dotsb <T_{j-1}\vert G_{j-1} < U\vert H < T_{j+1} \vert G_{j+1} <\dotsb < T_{\ell}\vert G_\ell)\,. \]
Once we know that $\mathcal{S}\vert \mathcal{F}< \mathcal{T}\vert\mathcal{G} \cup \{U\vert H\} \setminus \{T_j\vert G_j\}$ is lexicographically decreasing,
the claim follows from Lemma \ref{lem:canonical} and the fact that $\operatorname{rk}(U^c)< a-\ell \leq \operatorname{rk}(T_{j-1}^c)$. To check that $\mathcal{S}\vert \mathcal{F}< \mathcal{T}\vert\mathcal{G} \cup \{U\vert H\} \setminus \{T_j\vert G_j\}$ is lexicographically decreasing, we need to verify that $\min(\overline{S_s^c}\setminus H)\in G_{j-1}$. To see this, note that because $\mathcal{S}|\mathcal{F} < \mathcal{T}|\mathcal{G} \cup \{U|H\}$ is not lexicographically decreasing, $\min(\overline{S_s^c}\setminus H) = \min(\overline{S_s^c}\setminus G_j)$, so $\min(\overline{S_s^c}\setminus H) \in G_{j-1}$. 

We now show that $\operatorname{neg}(\mathbf{A}_j)$ is contained in the specified subset of $\operatorname{pos}(\mathbf{A}_{j+1})$. By Lemma \ref{lem:canonical}, the set $\operatorname{neg}(\mathbf{A}_j)$ consists of biflags
\[\mathcal{S}\vert \mathcal{F} < T_1' \vert G_1'< \dotsb <T_{j-1}'\vert G_{j-1}'<U'\vert H'  < T_{j}' \vert G_{j}' <T_{j+1}' \vert G_{j+1}'<\dotsb < T_{\ell}'\vert G_\ell' \]
that are lexicographically decreasing at indices $i\in[\ell+1]\setminus \{j-1,j\}$, not lexicographically decreasing at index $j$, and obey the condition that $T_j'\vert G_j'$ is the first biflat such that $\operatorname{rk}(T_j'^c)<a-\ell$. We write these biflags as $\mathcal{S}\vert \mathcal{F} < \mathcal{T}'\vert \mathcal{G}' \cup \{U'\vert H'\}$. First, observe that $\mathcal{S}\vert \mathcal{F} < \mathcal{T}'\vert \mathcal{G}' \cup \{U'\vert H'\}$ is actually lexicographically decreasing at index $j-1$. This follows from the facts that $G_{j}'\subset H'$, $\min(\overline{S_s^c}\setminus G_{j}') \not\in H'$, and $\min(\overline{S_s^c}\setminus G_{j}')\in G_{j-1}'$. We claim that $\mathcal{S}\vert \mathcal{F} < \mathcal{T}'\vert \mathcal{G}' \cup \{U'\vert H'\}$ is contained in the set
\[\operatorname{pos}(\mathcal{S}\vert \mathcal{F} < T_1' \vert G_1'< \dotsb <T_{j-1}'\vert G_{j-1}'<U'\vert H'  < T_{j+1}' \vert G_{j+1}' <\dotsb < T_{\ell}'\vert G_\ell'). \]
Once we know that $\mathcal{S}\vert \mathcal{F}< \mathcal{T}'\vert\mathcal{G}' \cup \{U'\vert H'\} \setminus \{T_j'\vert G_j'\}$ is lexicographically decreasing, the claim follows from Lemma \ref{lem:canonical} and the fact that $\operatorname{rk}(T_{j+1}'^c)< a-\ell \leq \operatorname{rk}(U'^c)$. To check that $\mathcal{S}\vert \mathcal{F}< \mathcal{T}'\vert\mathcal{G}' \cup \{U'\vert H'\} \setminus \{T_j'\vert G_j'\}$ is lexicographically decreasing, we only need to verify that $\min(\overline{S_s^c}\setminus G_{j+1}')\in H'$. This, in turn, follows from the fact that $G_j'\subset H'$ and that $\mathcal{S}\vert \mathcal{F} < \mathcal{T}'\vert \mathcal{G}'$ is lexicographically decreasing. 
\end{proof}

Let $\mathbf{B}\subseteq \operatorname{pos}(\mathbf{A}_1)$ be the subset of biflags $\mathcal{S}\vert\mathcal{F} < U\vert H <\mathcal{T}\vert \mathcal{G}$ such that $\overline{S_s^c} \subseteq H$. Note that the definition of $\operatorname{pos}(\mathbf{A}_1)$ ensures that $\operatorname{rk}(U^c)<a-\ell$. 

  \begin{lemma}\label{lem:total_cancel}
    Let $\mathbf{A}'$ denote the set of all lexicographically decreasing biflags whose first component is $\mathcal{S}\vert \mathcal{F}$ and whose second component has length $\ell+1$. We have the equality
    \[\mathbf{A'} = \left(\operatorname{pos}(\mathbf{A}_1)\setminus \mathbf{B}\right) \sqcup \bigsqcup_{j=1}^{\ell}\operatorname{pos}(\mathbf{A}_{j+1})\setminus\operatorname{neg}(\mathbf{A}_{j}) = \left(\operatorname{pos}(\mathbf{A})\setminus\mathbf{B}\right) \setminus \operatorname{neg}(\mathbf{A}). \]
  \end{lemma}
  \begin{proof}
    For an index $1\leq j \leq \ell+1$, let $\mathbf{A}'_j\subseteq \mathbf{A}'$ be the set of all biflags $\mathcal{S}\vert \mathcal{F} < \mathcal{T}' \vert \mathcal{G}'$ such that $j$ is the first index where $\operatorname{rk}(T_j'^c)<a-\ell$. Lemma \ref{lem:dyck} implies that $\operatorname{rk}(T_{\ell+1}'^c)< a-\ell$, so the collection $\{\mathbf{A}'_j: 1\leq j \leq \ell+1\}$ partitions $\mathbf{A}'$. 
    
    It follows from the definition of $\operatorname{pos}(\mathbf{A}_1)$ that $\operatorname{pos}(\mathbf{A}_1)\setminus \mathbf{B} = \mathbf{A}'_1$. For an index $1\leq j \leq \ell$, it follows from Lemma \ref{lem:neg_containment} that $\operatorname{pos}(\mathbf{A}_{j+1})\setminus\operatorname{neg}(\mathbf{A}_{j})= \mathbf{A}'_{j+1}$. That $\operatorname{neg}(\mathbf{A}_{\ell+1})$ is empty and does not need to be removed from the middle term is the content of Lemma \ref{lem:empty}.
  \end{proof}
  
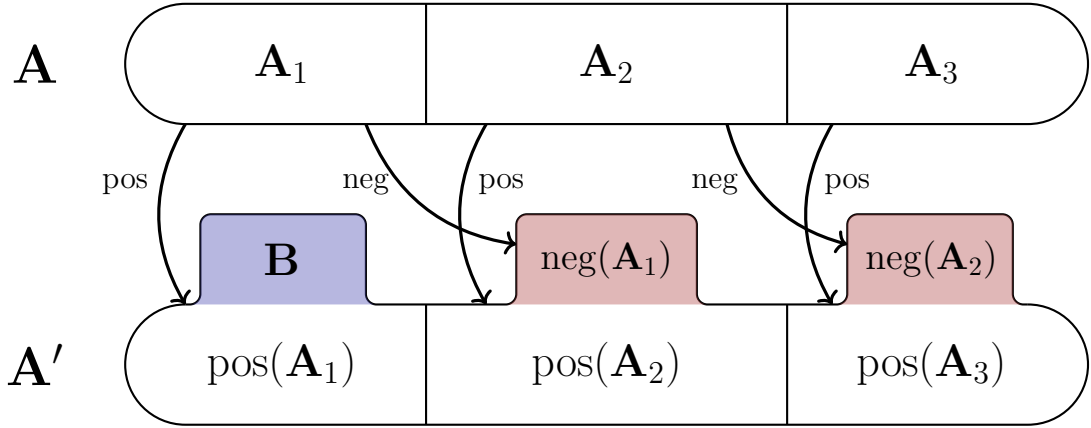
\begin{figure}
\begin{tikzpicture} [scale=.8]

\node at (-9.5, 0) {\Huge $\mathbf{A}$};
\draw[thick] (7,-1) arc (-90:90:1)--(-7,1) arc (-90:90:-1)-- cycle;
\draw[thick] (-3,1) to (-3,-1);
\draw[thick] (3,1) to (3,-1);
\node at (5.4,0) {\huge $\mathbf{A}_3$};
\node at (0,0) {\huge $\mathbf{A}_2$};
\node at (-5.4,0) {\huge $\mathbf{A}_1$};

\draw[thick]  (7,-4) arc (90:-90:1) --(-7,-6) arc (90:-90:-1);
\draw[thick] (-3,-6) to (-3,-4);
\draw[thick] (3,-6) to (3,-4);
\draw[thick] (7,-4)[rounded corners] -- (6.75,-4)[rounded corners] -- (6.75,-2.5)[rounded corners] -- (4,-2.5)[rounded corners] -- (4,-4)[rounded corners] -- (1.5,-4)[rounded corners] -- (1.5,-2.5)[rounded corners] -- (-1.5,-2.5)[rounded corners] -- (-1.5,-4)[rounded corners] -- (-4,-4)[rounded corners]-- (-4,-2.5)[rounded corners]--(-6.75,-2.5)[rounded corners]--(-6.75,-4)[rounded corners]--(-7,-4);
\fill[fill opacity =0.4, color=red!40!gray] (7,-4)[rounded corners] -- (6.75,-4)[rounded corners] -- (6.75,-2.5)[rounded corners] -- (4,-2.5)[rounded corners] -- (4,-4)[rounded corners] -- (1.5,-4)[rounded corners] -- (1.5,-2.5)[rounded corners] -- (-1.5,-2.5)[rounded corners] -- (-1.5,-4)[rounded corners] -- (-4.5,-4)--cycle;
\fill[fill opacity =0.4, color=blue!40!gray] (-1.5,-4)[rounded corners] -- (-4,-4)[rounded corners]-- (-4,-2.5)[rounded corners]--(-6.75,-2.5)[rounded corners]--(-6.75,-4)[rounded corners]--(-7,-4);

\node at (-9.5, -5) {\Huge $\mathbf{A'}$};
\node at (5.4,-5) {\huge $\operatorname{pos}(\mathbf{A}_3)$};
\node at (5.4,-3.25) {\LARGE $\operatorname{neg}(\mathbf{A}_2)$};

\node at (0,-5) {\huge $\operatorname{pos}(\mathbf{A}_2)$};
\node at (0,-3.25) {\LARGE $\operatorname{neg}(\mathbf{A}_1)$};

\node at (-5.4,-5) {\huge $\operatorname{pos}(\mathbf{A}_1)$};
\node at (-5.4,-3.25) {\huge $\mathbf{B}$};


\draw[very thick,arrows={->}] (-7,-1) to[bend right] (-7,-4);
\node at (-8,-2) {\Large$\operatorname{pos}$};
\draw[very thick,arrows={->}] (-4,-1) to[bend right] (-1.5,-3);
\node at (-4,-2) {\Large$\operatorname{neg}$};

\draw[very thick,arrows={->}] (-2,-1) to[bend right] (-2,-4);
\node at (-1.75,-2) {\Large$\operatorname{pos}$};
\draw[very thick,arrows={->}] (2,-1) to[bend right] (4,-3);
\node at (1.8,-2) {\Large$\operatorname{neg}$};

\draw[very thick,arrows={->}] (3.75,-1) to[bend right] (3.75,-4);
\node at (4,-2) {\Large$\operatorname{pos}$};
\end{tikzpicture}
\caption{An illustration of Lemma \ref{lem:total_cancel} in the $\ell+1=3$ case. The set $\operatorname{pos}(\mathbf{A})$ is equal to the entire bottom shape. The set $\mathbf{A}'$ is equal to the uncolored part of $\operatorname{pos}(\mathbf{A})$.}\label{fig:in_total}
\end{figure}

\begin{example}
    We illustrate some of the cancellation happening in Lemma~\ref{lem:total_cancel} for the pyramid graph $\Gamma$ of Figure~\ref{fig:graph}. The biflag $1267\vert E < E \vert 3$ has a second component of length one and is contained in the set $\mathbf{A}_1$. As we saw in Example~\ref{ex:canonical}, the biflag $1267\vert E < E \vert 34 < E \vert 3$ is in $\operatorname{pos}(1267\vert E < E\vert 3)$ and the biflag $1267\vert E < 124567\vert 378 < E\vert 3$ is in $\operatorname{neg}(1267\vert E < E\vert 3)$. As $1267\vert E < E \vert 34 < E \vert 3$ is lexicographically decreasing and has a second component of length $2$, it is contained in the set $\mathbf{A}'$. The biflag $1267\vert E < 124567\vert 378 < E\vert 3$ is not lexicographically decreasing, but it is contained in $\operatorname{pos}(1267\vert E < 124567\vert 378)$. As $1267\vert E < 124567\vert 378$ is in $\mathbf{A}_2$, this is one example of the containments $\operatorname{neg}(\mathbf{A}_1) \subseteq \operatorname{pos}(\mathbf{A}_2)$ and $\operatorname{pos}(\mathbf{A}_2)\setminus \operatorname{neg}(\mathbf{A}_1) \subseteq \mathbf{A}'$. 
\end{example}

  Applying Lemma \ref{lem:total_cancel} to the canonical expansion of \[ \sum_{\mathcal{S}\vert \mathcal{F} < \mathcal{T}\vert \mathcal{G}\in \mathbf{A}} x_{\mathcal{S}\vert \mathcal{F} < \mathcal{T}\vert \mathcal{G}} (\overline\gamma -v_{a-\ell}^-)\] 
  yields a nonnegative expression. The following proposition records what this nonnegative expression is.

    \begin{proposition}\label{prop:summary}
    Let $\mathbf{A}$ and $\mathbf{A}'$ denote the sets of all lexicographically decreasing biflags whose first component is $\mathcal{S}\vert \mathcal{F}$ and whose second component has length $\ell$ and $\ell+1$, respectively. Then
    \[\sum_{\mathcal{S}\vert\mathcal{F}<\mathcal{T}\vert\mathcal{G}\in \mathbf{A}}x_{\mathcal{S}\vert\mathcal{F}<\mathcal{T}\vert\mathcal{G}} (\overline\gamma -v_{a-\ell}^-) = \sum_{\mathcal{S}\vert\mathcal{F}<\mathcal{T}'\vert\mathcal{G}'\in \mathbf{A}'}x_{\mathcal{S}\vert\mathcal{F}<\mathcal{T}'\vert\mathcal{G}'} + \sum_{\mathcal{S}\vert \mathcal{F} < U\vert H < \mathcal{T}\vert \mathcal{G}\in \mathbf{B} }x_{\mathcal{S}\vert \mathcal{F} < U\vert H < \mathcal{T}\vert \mathcal{G}}\]
    where $\mathbf{B}$ is a collection of biflags $\mathcal{S}\vert \mathcal{F} < U\vert H < \mathcal{T}\vert \mathcal{G}$ such that $\overline{S_s^c}\subseteq H$ and $\operatorname{rk}(U^c)<a-\ell$.
  \end{proposition}
\begin{proof}
  By Lemma \ref{lem:disjoint} and Lemma~\ref{lem:total_cancel}, we have
    \begin{align*}
      \sum_{\mathbf{A}} x_{\mathcal{S}\vert \mathcal{F} < \mathcal{T}\vert \mathcal{G}} (\overline\gamma -v_{a-\ell}^-) &=\sum_{ \operatorname{pos}(\mathbf{A})} x_{\mathcal{S}\vert \mathcal{F} < \mathcal{T}\vert \mathcal{G}\cup \{U\vert H\}}-\sum_{\operatorname{neg}(\mathbf{A})} x_{\mathcal{S}\vert \mathcal{F} < \mathcal{T}\vert \mathcal{G}\cup \{U\vert H\}}\\
            &=\sum_{ \operatorname{pos}(\mathbf{A}_1)}x_{\mathcal{S}\vert \mathcal{F} < U\vert H < \mathcal{T}\vert \mathcal{G}} + \sum_{j=1}^\ell\sum_{ \operatorname{pos}(\mathbf{A}_{j+1})\setminus \operatorname{neg}(\mathbf{A}_j)}x_{\mathcal{S}\vert \mathcal{F} < \mathcal{T}\vert \mathcal{G}\cup \{U\vert H\}}\\
      &=\sum_{ \mathbf{A}'}x_{\mathcal{S}\vert \mathcal{F} < \mathcal{T}\vert \mathcal{G}\cup \{U\vert H\}} + \sum_{\mathbf{B} }x_{\mathcal{S}\vert \mathcal{F} < U\vert H < \mathcal{T}\vert \mathcal{G}}\,. \qedhere
  \end{align*}
\end{proof}

  We are now in a position to prove Proposition~\ref{thrm:min_dec}.
\begin{proof}[Proof of Proposition~\ref{thrm:min_dec}]
  We prove this by induction on the pair $(a,a-\ell)$. When $a=0$, $\mathcal{S}\vert \mathcal{F}$ is not a biflag and $x_{\mathcal{S}\vert \mathcal{F}}$ is zero.
  When $a-\ell=0$, Lemma \ref{lem:dyck} tells us both that $\operatorname{rk}(G_1 \cap \overline{S_s^c}) \leq a-1$ and that
  \[\operatorname{rk}(G_1\cap \overline{S_s^c})>\dotsb>\operatorname{rk}(G_{\ell-1}\cap \overline{S_s^c})>\operatorname{rk}(G_{\ell}\cap \overline{S_s^c}) >0. \]
  The chain of strict inequalities implies that $\operatorname{rk}(G_1\cap \overline{S_s^c})\geq a$, which is a contradiction. Therefore, there are no lexicographically decreasing biflags whose second component has length $a$ and our sum is zero.

  We now prove the claim for an arbitrary positive pair $(a,a-\ell)$ assuming the claim for all pairs $(a', a'-\ell')$ where either $a'<a$ and $a'-\ell' \leq a-\ell$ or $a' \leq a$ and $a'-\ell' < a-\ell$.

    Let $\mathbf{A}$ and $\mathbf{A}'$ denote the sets of all lexicographically decreasing biflags whose first component is $\mathcal{S}\vert \mathcal{F}$ and whose second component has length $\ell$ and $\ell+1$, respectively.  Proposition \ref{prop:summary} tells us that
  \begin{align*}
   & \sum_{\mathcal{S}\vert\mathcal{F}<\mathcal{T}\vert\mathcal{G}\in \mathbf{A}}x_{\mathcal{S}\vert\mathcal{F}<\mathcal{T}\vert\mathcal{G}}\left(\overline\gamma -v_{a-\ell}^-\right)\prod_{i=1}^{a-\ell-1} \left(\overline\gamma-v_i^-\right)\\
    &= \sum_{\mathcal{S}\vert \mathcal{F} < \mathcal{T}'\vert \mathcal{G}' \in \mathbf{A}'} x_{\mathcal{S}\vert \mathcal{F} < \mathcal{T}'\vert \mathcal{G}'}\prod_{i=1}^{a-\ell-1} \left(\overline\gamma-v_i^-\right)
      + \sum_{\mathcal{S}\vert \mathcal{F} < U\vert H < \mathcal{T}\vert \mathcal{G} \in \mathbf{B}}x_{\mathcal{S}\vert \mathcal{F} < U\vert H < \mathcal{T}\vert \mathcal{G}}\prod_{i=1}^{a-\ell-1} \left(\overline\gamma-v_i^-\right).
  \end{align*}
  The first sum is zero by the $(a,a-\ell-1)$ case of our induction hypothesis. By Proposition \ref{prop:summary}, every biflag indexing a term in the second sum has the property that $\overline{S_s^c} \subseteq H$ and $\operatorname{rk}(U^c)<a-\ell$, so $x_{\mathcal{S}\vert \mathcal{F} < U\vert H}\prod_{i=1}^{a-\ell-1} \left(\overline\gamma-v_i^-\right)$ vanishes by the induction hypothesis. Therefore the second sum is zero. 
\end{proof}

\begin{proposition}\label{prop:reductionToVs}
    In $A(\Sigma_{N,\mathrm{M}})$, we have \[\delta^r + c_1 \delta^{r-1} + \dotsb + c_r = (\delta + u_1) \dotsb (\delta + u_r) = (\overline{\gamma} - v_1^-) \dotsb (\overline{\gamma} - v_r^-).\]
\end{proposition}
\begin{proof}
The first equality follows from Lemma \ref{lem:uReduction}. Note that $v_1^+ = 0$. By the $\ell=0$ case of Proposition~\ref{thrm:min_dec}, $v_j^+ \prod_{i=1}^{j-1} (\overline{\gamma} - v_i^-) = 0$ in $A(\Sigma_{N, \mathrm{M}})$ for each $j\in [r]$, so the claim follows.
\end{proof}

\begin{proof}[Proof of Theorem~\ref{thm:identity}]
It is known that $\Sigma_{N, N}$ is a projective fan \cite[Proposition 2.20]{ADH}, so $\Sigma_{N, \mathrm{M}}$ is quasi-projective. As the support of $\Sigma_{N, \mathrm{M}}$ is the product of $\mathbb{R}^N$ with the support of the Bergman fan of $\mathrm{M}$, it is a Lefschetz fan by Proposition~\ref{prop:Lefschetzfans}, so $A(\Sigma_{N, \mathrm{M}})$ satisfies Poincar\'{e} duality. 

By Proposition \ref{prop:reductionToVs}, it suffices to show that $(\overline{\gamma} - v_1^-) \dotsb (\overline{\gamma} - v_r^-)=0$ in $A(\Sigma_{N,\mathrm{M}})$.
The balanced fan $\Sigma_{N,\mathrm{M}}$ defines a Minkowski weight $[\Sigma_{N, \mathrm{M}}]$ in $A(\Sigma_{N,N})$. By Poincar\'{e} duality for $A(\Sigma_{N, \mathrm{M}})$ and because the restriction map $\iota^* \colon A(\Sigma_{N,N}) \to A(\Sigma_{N, \mathrm{M}})$ is surjective, $(\overline{\gamma} - v_1^-) \dotsb (\overline{\gamma} - v_r^-) = 0$ if and only if $\deg_{\Sigma_{N, \mathrm{M}}}(\overline{\gamma} - v_1^-) \dotsb (\overline{\gamma} - v_r^-) \iota^*z = 0$ for every $z \in A(\Sigma_{N, N})$. Using the identity 
$$\deg_{\Sigma_{N, \mathrm{M}}}(\overline{\gamma} - v_1^-) \dotsb (\overline{\gamma} - v_r^-) \iota^*z = \deg_{\Sigma_{N,N}} (\overline{\gamma} - v_1^-) \dotsb (\overline{\gamma} - v_r^-) [\Sigma_{N, \mathrm{M}}] z$$
and Poincar\'{e} duality for $A(\Sigma_{N,N})$, we see that 
$$(\overline{\gamma} - v_1^-) \dotsb (\overline{\gamma} - v_r^-) = 0 \text{ in }A(\Sigma_{N, \mathrm{M}}) \text{ if and only if }(\overline{\gamma} - v_1^-) \dotsb (\overline{\gamma} - v_r^-)[\Sigma_{N, \mathrm{M}}] = 0 \text{ in }A(\Sigma_{N,N}).$$
Let $T \mathrm{M}$ denote the truncation of $\mathrm{M}$; this is a matroid whose bases are the independent sets of $\mathrm{M}$ of rank $r-1$. If $r > 1$, then $T \mathrm{M}$ is loopless because $\mathrm{M}$ is. By \cite[Proposition 4.4]{TropicalLin} (see also \cite[Corollary 1.7]{EHL}), if $r > 1$, then $\overline{\gamma} \cdot [\Sigma_{N, \mathrm{M}}] = [\Sigma_{N, T \mathrm{M}}]$. If $r=1$, then $\overline{\gamma} \cdot [\Sigma_{N, \mathrm{M}}] = 0$. 

We induct on the rank of $\mathrm{M}$. As $\mathrm{M}$ is assumed to be loopless, the least possible rank of $\mathrm{M}$ is $1$. The argument we will give for the inductive step works when $\mathrm{M}$ has rank $1$, using that in this case $\overline{\gamma} \cdot [\Sigma_{N, \mathrm{M}}] = 0$ as input. 

Note that $\iota^*v_r^- = 0$ in $A(\Sigma_{N, \mathrm{M}})$: given a term $x_{S|F}$ with $\operatorname{rk}(S^c) \ge r$, then we must have $\operatorname{rk}(F) = r$ because $S^c \subseteq F$. But this implies that $F = [N]$. It follows that $v_r^{-} [\Sigma_{N, \mathrm{M}}] = 0$. We therefore have
$$(\overline{\gamma} - v_1^-) \dotsb (\overline{\gamma} - v_r^-)[\Sigma_{N, \mathrm{M}}] =  (\overline{\gamma} - v_1^-) \dotsb (\overline{\gamma} - v_{r-1}^-)[\Sigma_{N, T \mathrm{M}}].$$
Note that, for $j < r$, the classes $v_j^-$ defined using $\mathrm{M}$ are the same as the classes $v_j^-$ defined using $T\mathrm{M}$: the rank of $S^c$ in $\mathrm{M}$ is at least $j$ if and only if the rank of $S^c$ in $T \mathrm{M}$ is at least $j$. The desired vanishing then follows from induction. 
\end{proof}

\section{Proof of the main result}

We now proceed with the proof of Theorem~\ref{thm:kahler}. We first establish some results on rings which resemble the cohomology rings of projective bundles over smooth projective varieties. 

\subsection{Projective bundle rings}\label{ssec:pbundle}

Let $A$ be a finite-dimensional commutative graded $\mathbb{R}$-algebra equipped with an open convex cone $\mathcal{K} \subset A^1$ and a map $\deg_A \colon A^n \to \mathbb{R}$ which satisfies the K\"{a}hler package. We will develop the basic algebraic properties of rings which resemble the cohomology rings of projective bundles. 
For some $r \ge 1$, let $c_1, \dotsc, c_r$ be classes in $A^1, \dotsc, A^r$, respectively. Set
$$B = \frac{A[\zeta]}{(\zeta^r + c_1 \zeta^{r-1} + \dotsb + c_r)}.$$
There is a natural ring homomorphism $\pi^* \colon A \to B$. As $A$-modules, we have a decomposition
\begin{equation}\label{eq:decomposition}
  B \cong \bigoplus_{i=0}^{r-1} \zeta^i A.
\end{equation}
There is an isomorphism $\deg_B \colon B^{n + r - 1} \to \mathbb{R}$ which is defined by the condition that if $x \in A^n$ is a class with $\deg_A(x) = 1$, then $\deg_B(\zeta^{r-1} x) = 1$. 
\begin{proposition}\label{prop:PD}
The ring $B$ has Poincar\'{e} duality. 
\end{proposition}
\begin{proof}
For each $k \ge 0$, by \eqref{eq:decomposition}, $B^k \cong \bigoplus_{i=0}^{r-1} \zeta^i \cdot A^{k - i}$ and $B^{n + r - 1 - k} \cong \bigoplus_{j=0}^{r-1} \zeta^{j} \cdot A^{n + r - 1 - k - j}.$ The pairing between two terms in these decompositions is zero unless $i + j \ge r-1$, and it is nondegenerate when $j = r - 1 - i$. After choosing a basis for each $A^i$, the matrix representing the Poincar\'{e} pairing is block triangular, and therefore this pairing is nondegenerate by Poincar\'{e} duality for $A$. 
\end{proof}

Theorem~\ref{thm:kahler} establishes cases in which the K\"{a}hler package holds for $B$. In such cases, an argument from \cite{BlochGieseker} gives consequences for $A$. 

\begin{theorem}\label{thm:blochgieseker}
Suppose that, for each $i$, the map $B^i \to B^{i+1}$ given by multiplication by $\zeta$ has full rank. Let $d = \min\{r, n\}$. Then the map $A^i \to A^{i+d}$ given by multiplication by $c_d$ is injective for $i \le (n-r)/2$ and is surjective for $i \ge (n-r)/2$. 
\end{theorem}

While the original proof in \cite{BlochGieseker} can be adapted to this setting, we give a slightly different argument. 

\begin{proof}
First consider the case when $r \le n$. Under the Poincar\'{e} pairing, the map $A^i \to A^{i + r}$ given by multiplication by $c_r$ is dual to the map $A^{n - i - r} \to A^{n - i}$ given by multiplication by $c_r$, so it suffices to prove the injectivity part. By \eqref{eq:decomposition}, we have $B^{r + i - 1} = \zeta^{r-1} A^{i} \oplus \zeta^{r-2} A^{i+1} \oplus \dotsb \oplus A^{r + i - 1}$ and $B^{r + i} = \zeta^{r-1} A^{i + 1} \oplus \dotsb \oplus A^{r + i}$. As $i \le (n - r)/2$, the map $B^{r + i - 1} \to B^{r + i}$ given by multiplication by $\zeta$ is injective. Comparing with these decompositions, we see that the map $A^{i} \to A^{r + i}$ given by multiplication by $c_r$ is injective. 

When $r \ge n$, the only nontrivial statement is that $c_n$ is nonzero. This follows from the fact that the map $B^{r-1} \to B^r$ given by multiplication by $\zeta$ is injective. 
\end{proof}

Theorem~\ref{thm:kahler} does not assert that multiplication by $\zeta$ satisfies the hypothesis of Theorem~\ref{thm:blochgieseker}. However, one can apply Theorem~\ref{thm:blochgieseker} to the rings considered in Theorem~\ref{thm:kahler} by ``twisting.'' Given $\delta \in A^1$, for each $i$ set $c_i' = \sum_{j=0}^{i} (-1)^j\binom{r - i + j}{j} c_{i-j} \delta^j$. Then there is an isomorphism of $A$-algebras
$$\frac{A[\zeta]}{(\zeta^r + c_1' \zeta^{r-1} + \dotsb + c_r')} \to B, \quad \zeta \mapsto \zeta + \delta.$$
In particular, if $A$ is as in Theorem~\ref{thm:kahler} and $\delta \in \mathcal{K}$, then the projective bundle ring defined using $c_1', c_2', \dotsc$ has the K\"{a}hler package with respect to a cone which contains $\zeta$ in its interior, so Theorem~\ref{thm:blochgieseker} implies that multiplication by $c_d'$ has maximal rank. 

Suppose that $B$ has the K\"{a}hler package with respect to the cone $\mathcal{K} + \mathbb{R}_{>0} \zeta$, and that $r \ge n$. For $\delta \in \mathcal{K}$ and a real number $\lambda$, if we twist by $\lambda \delta$, we have $\deg_A(c_n') = (-1)^n\lambda^{n}\binom{r}{n}\deg_A(\delta^n) + O(\lambda^{n-1})$. This has sign $(-1)^n$ for $\lambda$ sufficiently large. As Theorem~\ref{thm:blochgieseker} implies that $c_n'$ does not vanish after we twist by $\lambda \delta$ for any $\lambda > 0$, we have the following corollary. 

\begin{corollary}
Suppose that $r \ge n$, and $B = A[\zeta]/(\zeta^r + c_1 \zeta^{r-1} + \dotsb + c_n \zeta^{r - n})$ has the K\"{a}hler package with respect to the cone $\mathcal{K} + \mathbb{R}_{>0} \zeta$. Then $(-1)^n\deg_A(c_n) \ge 0$. 
\end{corollary}

If $B$ has the K\"{a}hler package with respect to the cone $\mathcal{K} + \mathbb{R}_{>0} \zeta$, then we can show that certain quotients of $A$ have the K\"{a}hler package. Given classes $c_1, \dotsc, c_r$, define classes $s_i \in A^i$ for each $i$ via the identity
\begin{equation}\label{eq:segre}
(1 + c_1 + \dotsb + c_r) (1 + s_1 + \dotsb + s_n) = 1.
\end{equation}
For example, $s_1 = -c_1$ and $s_2 = c_1^2 - c_2$. 

\begin{proposition}\label{prop:kahlerquotient}
Suppose that $B$ has the K\"{a}hler package with respect to the cone $\mathcal{K} + \mathbb{R}_{>0} \zeta$. Suppose that $s_t \not=0$, and that $s_{j} = 0$ for $j > t$. Then $A/\operatorname{ann}(s_t)$ has the K\"{a}hler package with respect to the image of $\mathcal{K}$ and the degree map induced by $A$. 
\end{proposition}

\begin{proof}
Let $\pi^* \colon A \to B$ be the natural inclusion. As both $A$ and $B$ have Poincar\'{e} duality, $\pi^*$ is dual to a map of $A$-modules $\pi_* \colon B^{\bullet} \to A^{\bullet - r + 1}$. In terms of \eqref{eq:decomposition}, if we write an element $x \in B^i$ as $a_i + a_{i-1} \zeta + \dotsb + a_{i-r + 1} \zeta^{r-1}$, then $\pi_* x = a_{i-r+1}$. We see that an element $x$ in $B$ vanishes if and only if $\pi_* \zeta^i x = 0$ for all $i \ge 0$. It follows from repeatedly applying the defining relation of $B$ and using Pieri's rule that $\pi_* \zeta^i = s_{i - r + 1}$ for $i \ge r - 1$. 

Because $s_j = 0$ for $j > t$, it follows that $\zeta^{t + r} = 0$. This implies that the map $A \to B/\operatorname{ann}(\zeta^{t + r - 1})$ is surjective, as $\zeta \in \operatorname{ann}(\zeta^{t + r - 1})$. Also, an element $a \in A$ lies in $\operatorname{ann}(s_t)$ if and only if $\pi^* a$ lies in $\operatorname{ann}(\zeta^{t + r - 1})$. This induces an isomorphism from $A/\operatorname{ann}(s_t)$ to $B/\operatorname{ann}(\zeta^{t + r - 1})$. The result follows from the descent lemma \cite[Lemma 1.16]{CKS87} \cite[Theorem 2.1.5]{KK87}, see \cite[Theorem 3.2]{Cat08} or \cite[Lemma 4.6]{LefschetzModule}, as $\zeta$ lies in the closure of $\mathcal{K} + \mathbb{R}_{>0} \zeta$.
\end{proof}

\begin{remark}\label{rem:somecone}
We note that, if $A$ has the K\"{a}hler package, then there will exist a cone $\mathcal{K}_B \subset B^1$ such that $B$ has the K\"{a}hler package with respect to $\mathcal{K}_B$ and $\deg_B$. Indeed, consider the free $\mathbb{R}[z]$-module given by 
$$B_z = \frac{A[\zeta, z]}{(\zeta^r + c_1 \zeta^{r-1}z + \dotsb  + c_r z^r)}.$$
For each $z_0 \in \mathbb{R}$, let $B_{z_0}$ be the ring obtained by setting $z = z_0$, so $B_1$ is identified with $B$. 
It is known that $B_0$ has the K\"{a}hler package with respect to the cone spanned by $\mathcal{K}$ and $\zeta$, see, e.g., \cite[Section 7.2]{AHK18}. Therefore, for any $\ell \in \mathcal{K}$, Hard Lefschetz and the Hodge--Riemann relations hold for $\ell + \zeta$ in $B_0$. This implies that Hard Lefschetz and the Hodge--Riemann relations hold for $\ell + \zeta$ in $B_{\varepsilon}$ for all sufficiently small $\varepsilon$. For any $\varepsilon > 0$, there is an isomorphism $B_{\varepsilon} \to B_1$ which sends $\zeta$ to $\zeta/\varepsilon$. The induced isomorphism $B_{\varepsilon}^{n + r - 1} \to B_1^{n + r - 1}$ scales the degree map by a positive scalar, and so $\ell + \zeta/\varepsilon$ satisfies Hard Lefschetz and the Hodge--Riemann relations in $B_1$. 
\end{remark}

However, the argument in Remark~\ref{rem:somecone} gives no control over the cone $\mathcal{K}_B$, and so it is not useful for most applications.

\subsection{Proof of Theorem~\ref{thm:kahler}}

We now proceed with the proof of Theorem~\ref{thm:kahler}. 
We first reduce to the case when $\mathrm{M}$ is loopless. Let $S \subset [N]$ be the set of loops of $\mathrm{M}$, and let $\mathrm{M}'$ be the matroid $\mathrm{M} \setminus S$. Define $\varphi' \colon \mathbb{R}^{[N] \setminus S} \to \mathbb{R}^m$ as the restriction of $\varphi^*$ to the subspace where the coordinates labeled by $S$ are $0$. Then $\varphi^*(P(\mathrm{M})) = \varphi'(P(\mathrm{M}'))$, and the classes associated to $\mathrm{M}'$ are the same as the classes associated to $\mathrm{M}$. 

For the remainder, we will assume that $\mathrm{M}$ is loopless. 
Recall that $\Sigma_N$ denotes the permutohedral fan in $\mathbb{R}^N$. 
Let $\tilde{\Sigma}$ be a simplicial projective refinement of $\Sigma \times \Sigma_N$ which has the property that the map $\iota \colon \mathbb{R}^m \times \mathbb{R}^N \to \mathbb{R}^N \times \mathbb{R}^N$ given by $\begin{pmatrix} \varphi_* & 0 \\ 0 & I \end{pmatrix}$ extends to a map of fans $\tilde{\Sigma} \to \Sigma_{N, N}$. Because $\tilde{\Sigma}$ is a refinement of $\Sigma \times \Sigma_N$, there is a subfan $\Pi$ of $\tilde{\Sigma}$ with the same support as $\mathbb{R}^m \times \Sigma_{\mathrm{M}}$.

\begin{lemma}\label{lem:Pilefschetz}
$\Pi$ is a Lefschetz fan. 
\end{lemma}

\begin{proof}
$\Pi$ is constructed as a subfan of a projective fan, so it is quasiprojective. 
Its support is the same as $\mathbb{R}^m \times \Sigma_{\mathrm{M}}$, which is a Lefschetz fan by Proposition~\ref{prop:Lefschetzfans}. 
\end{proof}

The map $\Pi \to \Sigma$ induced by projection gives an injective pullback map $A(\Sigma) \to A(\Pi)$, identifying $A(\Sigma)$ with a subring of $A(\Pi)$. In particular, we have elements $c_1, \dotsc, c_r$ in $A(\Pi)$. 
Let $H$ denote the pullback of $\delta$ to $A(\Pi)$ along the map $\Pi \to \Sigma_{N,N}$ induced by $\iota$. 

\begin{proposition}\label{prop:subring}
The subring of $A(\Pi)$ generated by $H$ and $A(\Sigma)$ is isomorphic to $B$ via the map which sends $\zeta$ to $[H]$. 
\end{proposition}
 
\begin{proof}
We first check that $H^r + c_1 H^{r-1} + \dotsb + c_r = 0$. The image of the map $\Pi \to \Sigma_{N, N}$ is $\Sigma_{N, \mathrm{M}}$. 
Note that $H$ is the pullback of $\delta$ from $\Sigma_{N, \mathrm{M}}$, and that $c_i$ is the pullback of the class $c_i \in A^{i}(\Sigma_{N, \mathrm{M}})$. The result follows from Theorem~\ref{thm:identity}.

This implies that $B$ has a surjective map to the subring of $A(\Pi)$ generated by $H$ and $A(\Sigma)$. Proposition~\ref{prop:PD} and the fact that $A(\Sigma)$ has Poincar\'{e} duality imply that $B$ has Poincar\'{e} duality. 
It follows that every ideal of $B$ intersects $B^{n+r-1}$. So to check that the map from $B$ to $A(\Pi)$ is injective, it suffices to check that the image of $B^{n + r - 1}$ is nonzero. 

Choose an ample divisor $D$ in $A^1(\Sigma)$, and let $\tilde{D}$ be its pullback to $A^1(\Pi)$. We check that $\tilde{D}^n \cdot H^{r - 1}$ is nonzero in $A^{n + r- 1}(\Pi)$. Pushing forward to $A(\tilde{\Sigma})$, this is equivalent to checking that $\tilde{D}^{n} \cdot H^{r-1} \cdot [\Pi]$ is nonzero, where $[\Pi]$ is the Minkowski weight corresponding to the balanced fan $\Pi$. Because $D^n$ is a positive multiple of the class of a point, $\tilde{D}^n \cdot [\Pi]$ is a positive multiple of $[\{0\} \times \Sigma_{\mathrm{M}}]$, where $\{0\} \times \Sigma_{\mathrm{M}}$ is the subfan of $\tilde{\Sigma}$ with support $\{0\} \times \Sigma_{\mathrm{M}}$. The restriction of $H$ to the Chow ring of this fan is identified with the pullback of $\alpha$ from $A^1(\Sigma_{\mathrm{M}})$, so the claim follows from the fact that $\alpha^{r-1}$ is nonzero in $A(\Sigma_{\mathrm{M}})$ \cite[Proposition 5.8(2)]{AHK18}. 
\end{proof}

We will need the following result to prove Theorem~\ref{thm:kahler}.

\begin{proposition}\label{prop:AHLcriterion}\cite[Proposition 2.10]{LefschetzModule}
Let $A$ be a finite-dimensional commutative graded $\mathbb{R}$-algebra equipped with an open convex cone $\mathcal{K} \subset A^1$ and a map $\deg_A \colon A^n \to \mathbb{R}$ which satisfies the K\"{a}hler package. Let $B$ be a subalgebra of $A$ which is generated by a subset of $\overline{\mathcal{K}}$ that has $B^n = A^n$, and which satisfies Poincar\'{e} duality with respect to $\deg_A$. Then $B$ has the K\"{a}hler package with respect to the cone $(\overline{\mathcal{K}} \cap B^1)^{\circ}$. 
\end{proposition}

\begin{proof}[Proof of Theorem~\ref{thm:kahler}]
By Lemma~\ref{lem:Pilefschetz}, $A(\Pi)$ has the K\"{a}hler package with respect to a cone $\mathcal{K}_{\Pi}$ whose closure contains the classes of any convex piecewise linear function on $\mathbb{R}^m \times \mathbb{R}^N$ that is linear on the cones of $\Pi$. In particular, the pullback of any class in $\mathcal{K} \subset A^1(\Sigma)$ to $A^1(\Pi)$ is in the closure of this cone, and $[H]$ is in the closure of this cone. Because $\Sigma$ is a projective fan, $A^1(\Sigma)$ is spanned by classes of convex piecewise linear functions, which pull back to classes in $\overline{\mathcal{K}}_{\Pi}$. The Chow ring of a simplicial toric variety is generated in degree $1$, so $A(\Sigma)$, viewed as a subring of $A(\Pi)$, is generated by elements in $\overline{\mathcal{K}}_{\Pi}$. 
By Proposition~\ref{prop:subring}, $B$ is a subring of $A(\Pi)$ generated by classes in $\overline{\mathcal{K}}_{\Pi}$. 
As $B$ is generated by classes in $\overline{\mathcal{K}}_{\Pi}$ and
has Poincar\'{e} duality, Proposition~\ref{prop:AHLcriterion} implies that $B$ has the K\"{a}hler package with respect to the interior of $\overline{\mathcal{K}}_{\Pi} \cap B^1$, proving the result. 
\end{proof}

\medskip

We now mention a result which strengthens Theorem~\ref{thm:kahler} in two different ways. We allow one to take the projectivization of multiple projective bundles, and we also allow the base to be replaced by an arbitrary Lefschetz fan. Let $\Sigma$ be a Lefschetz fan in $\mathbb{R}^m$, and let $\mathrm{M}_1, \dotsc, \mathrm{M}_k$ be matroids on $[N_1], \dotsc, [N_k]$ of ranks $r_1, \dotsc, r_k$, respectively. For each $i$, let $\varphi_i^* \colon \mathbb{Z}^{N_i} \to \mathbb{Z}^m$ be a linear map, and assume that the transpose $\varphi_{i*}$ induces a map of fans from $\Sigma$ to the normal fan of $P(\mathrm{M}_i)$. 
For each $i$ and $j$, let $c_j(\mathrm{M}_i)$ denote the associated class in $A^j(\Sigma)$. 
Consider the ring 
$$B = \frac{A(\Sigma)[\zeta_1, \dotsc, \zeta_k]}{(\zeta^{r_1}_1 + c_1(\mathrm{M}_1) \zeta^{r_1 - 1}_1 + \dotsb + c_{r_1}(\mathrm{M}_1), \dotsc, \zeta^{r_k}_k + c_1(\mathrm{M}_k) \zeta^{r_k - 1}_k + \dotsb + c_{r_k}(\mathrm{M}_k))}.$$
Then $B$ is equipped with an isomorphism $\deg_B$ induced by its structure as an iterated projective bundle ring (see Section~\ref{ssec:pbundle}). Set $\mathcal{K}_B$ to be the open convex cone in $B^1$ generated by the interior of the pullback of $\mathcal{K}(\Sigma)$ and the classes $\zeta_1, \dotsc, \zeta_k$. 

\begin{theorem}\label{thm:multiplebundle}
The ring $B$ has the K\"{a}hler package with respect to $\mathcal{K}_B$ and $\deg_B$. 
\end{theorem}

\begin{proof}
We can reduce to the case when all of the $\mathrm{M}_i$ are loopless. 
As in the proof of Theorem~\ref{thm:kahler}, construct a quasi-projective fan $\Pi$ in $\mathbb{R}^m \times \mathbb{R}^{N_1} \times \dotsb \times \mathbb{R}^{N_k}$ whose support is the product of the support of $\Sigma$ with the support of the Bergman fans of the $\mathrm{M}_i$. One can choose a simplicial fan structure on $\Pi$ so that:
\begin{enumerate}
\item the projection of $\Pi$ to $\mathbb{R}^m$ is a map of fans with image $\Sigma$, and
\item for each $i$, the map $\mathbb{R}^m \times \mathbb{R}^{N_1} \times \dotsb \times \mathbb{R}^{N_k} \to \mathbb{R}^{N_i} \times \mathbb{R}^{N_i}$ induced by $\varphi_{i*}$ gives a map of fans from $\Pi$ to the projective bundle fan $\Sigma_{N_i, \mathrm{M}_i}$. 
\end{enumerate} 
As in the proof of Proposition~\ref{prop:subring}, the subring of $A(\Pi)$ generated by $A(\Sigma)$ and the pullbacks of the various $\delta$ classes from each $\Sigma_{N_i, \mathrm{M}_i}$ is isomorphic to the ring $B$; that the relations in the definition of $B$ hold follows from Theorem~\ref{thm:identity}. Note that $\Pi$ is a Lefschetz fan, and so $A(\Pi)$ has the K\"{a}hler package. 
As $B$ has Poincar\'{e} duality by Proposition~\ref{prop:PD}, the result follows from Proposition~\ref{prop:AHLcriterion}. 
\end{proof}

For example, taking $\Sigma$ to be the Bergman fan of a matroid, we have the following corollary of Theorem~\ref{thm:multiplebundle}. 

\begin{example}\label{thm:matroidbundle}
Let $\mathrm{N}$ be a loopless matroid on $[n]$, and let $\mathrm{M}_1, \dotsc, \mathrm{M}_k$ be matroids on $[n]$ of ranks $r_1, \dotsc, r_k$, respectively. For each $i$ and $j$, let $c_j(\mathcal{S}_{\mathrm{M}_i})$ denote the restriction of the class defined in \cite{BEST} to $A^j(\mathrm{N})$. 
Then the ring 
$$\frac{A(\mathrm{N})[\zeta_1, \dotsc, \zeta_k]}{(\zeta^{r_1}_1 + c_1(\mathcal{S}_{\mathrm{M}_1}) \zeta^{r_1 - 1}_1 + \dotsb + c_{r_1}(\mathcal{S}_{\mathrm{M}_1}), \dotsc, \zeta^{r_k}_k + c_1(\mathcal{S}_{\mathrm{M}_k}) \zeta^{r_k - 1}_k + \dotsb + c_{r_k}(\mathcal{S}_{\mathrm{M}_k}))}$$
has the K\"{a}hler package with respect to the interior of the cone generated by restrictions of ample divisors from the permutohedral fan and $\zeta_1, \dotsc, \zeta_k$. 
\end{example}

\medskip

In a key case, the ring considered in Theorem~\ref{thm:kahler} is the Chow ring of a smooth projective toric variety. However, the cone $\mathcal{K}$ appearing in Theorem~\ref{thm:kahler} is strictly larger  than the nef cone of this toric variety, and so this observation cannot be used to prove Theorem~\ref{thm:kahler} in this case. 

\begin{example}\label{ex:nonnef}
Let $\Sigma$ be the permutohedral fan, and let $\mathrm{M}$ be a (possibly non-realizable) matroid. Let $c_i \in A^i(\Sigma)$ be the class denoted $c_i(\mathcal{S}_{\mathrm{M}})$ in \cite{BEST}. In \cite[Appendix III]{BEST}, it is shown that there are line bundles $\mathcal{L}_1, \dotsc, \mathcal{L}_r$ on the permutohedral toric variety such that $c_i$ is the $i$th Chern class of $\mathcal{L}_1 \oplus \dotsb \oplus \mathcal{L}_r$. Then the projectivization of the direct sum of line bundles $\mathcal{L}_1 \oplus \dotsb \oplus \mathcal{L}_r$ over the permutohedral toric variety has the structure of a toric variety, and its Chow ring is $A(\Sigma)[\zeta]/(\zeta^r + c_1 \zeta^{r-1} + \dotsb + c_r)$. However, the vector bundle $\mathcal{L}_1 \oplus \dotsb \oplus \mathcal{L}_r$ is essentially never nef, so the divisor class corresponding to $\zeta$ is essentially never nef on this toric variety. 
\end{example}

The proof of Theorem~\ref{thm:kahler} relies crucially on the main result of \cite{AHK18}, which proves the K\"{a}hler package for the Chow ring of matroids. However, we note that the special case of Theorem~\ref{thm:kahler} mentioned in Example~\ref{ex:nonnef} easily implies the main result of \cite{AHK18}. 

\begin{example}\label{ex:AHK}
Let $\Sigma$ be the permutohedral fan, and let $\mathrm{M}$ be a loopless matroid. Let $c_i \in A^i(\Sigma)$ be the class denoted $c_i(\mathcal{S}_{\mathrm{M}})$ in \cite{BEST}, and let $B = A(\Sigma)[\zeta]/(\zeta^r + c_1 \zeta^{r-1} + \dotsb + c_r)$. The classes $s_1, s_2, \dotsc$ defined using \eqref{eq:segre} are equal to the classes denoted $c_1(\mathcal{Q}_{\mathrm{M}}), c_2(\mathcal{Q}_{\mathrm{M}}), \dotsc, c_{n-r}(\mathcal{Q}_{\mathrm{M}})$ in \cite{BEST}. In particular, $s_j = 0$ for $j > n-r$. 
It follows from Theorem~\ref{thm:kahler} and Proposition~\ref{prop:kahlerquotient} that $A(\Sigma)/\operatorname{ann}(c_{n-r}(\mathcal{Q}_{\mathrm{M}}))$ has the K\"{a}hler package with respect to the image of the cone generated by ample divisors on the permutohedral toric variety. 
There is an identification of $A(\Sigma)/\operatorname{ann}(c_{n-r}(\mathcal{Q}_{\mathrm{M}}))$ with the Chow ring of a matroid; this follows from Poincar\'{e} duality for Chow rings of matroids, see \cite[Theorem 4.2.1]{BES}. This then implies \cite[Theorem 1.4]{AHK18}. 
\end{example}

\bibstyle{alpha}
\bibliography{matroid.bib}

\end{document}